\documentclass[final, 12pt,reqno]{amsart}
\usepackage{amssymb,amsmath,graphicx,amsfonts,euscript}
\usepackage{color}
\usepackage[pagewise]{lineno}
\usepackage{showkeys}
\usepackage{fancyhdr}
\usepackage[margin=1in]{geometry}
\usepackage{titletoc}
\usepackage{appendix}
\usepackage{color}
\usepackage{hyperref}
\usepackage{float}

\newtheorem{theorem}{\bf Theorem}[section]

\newtheorem{proposition}[theorem]{\bf Proposition}

\newtheorem{remark}{\bf Remark}[section]

\newtheorem{lemma}[theorem]{\bf Lemma}

\numberwithin{equation}{section}
\allowdisplaybreaks[4]

\newcommand{\beq}{\begin{equation}}
    \newcommand{\eeq}{\end{equation}}
\newcommand{\ben}{\begin{eqnarray}}
    \newcommand{\een}{\end{eqnarray}}
\newcommand{\beno}{\begin{eqnarray*}}
    \newcommand{\eeno}{\end{eqnarray*}}

\begin{document}

\title[Stability threshold for 2-D Taylor-Couette flow]{Nonlinear Stability of Taylor-Couette Flows with
Heat Buoyancy}

\author[Yeping Li Gaofeng Wang and Tianfang Wu ]{\sc Yeping Li, Gaofeng Wang and Tianfang Wu}



\address{  School of Mathematics and Statistics, Nantong University, Nantong, 226019,  China.}
\email{ypleemei@aliyun.com}

\address{  School of Mathematics and Statistics, Nantong University, Nantong, 226019,  China.}
\email{gfwang@ntu.edu.cn}

\address{Wen zhou Business College, Wen zhou 325000, P. R. China.}
\email{20249237@wzbc.edu.cn}

    \bibliographystyle{abbrv}
    \maketitle
    \vskip .01in
    \begin{center}
        \sc Abstract
    \end{center}

    \vskip .05in
This paper investigates the nonlinear stability of Taylor-Couette
(TC) flows incorporating the thermal buoyancy within an annular
domain characterized by small viscosity $\nu$ and thermal
diffusivity $\mu$. It is well established that the buoyancy induced
convection significantly impacts practical industrial applications
of Taylor-Couette flow \cite{Chen2006}. In contrast to
\cite{An.2024},  we specifically examines the influence of the
temperature gradients and the gravity on the stability of
Taylor-Couette flows in this article. The thermal buoyancy term
introduces a destabilizing radial derivative $\partial_r$ into the
rotating TC system. To mitigate this destabilizing effect, we employ
estimates involving the negative derivatives. Consequently, the
additional viscous damping becomes necessary to counterbalance the
buoyancy induced instability. Our stability criterion requires that
the initial perturbations from the Taylor-Couette flow are bounded
by a suitable power of the viscosity. Under this condition, we prove
that solutions to the 2D Boussinesq system on $[1, R] \times
\mathbb{S}^1$ remain close to the Taylor-Couette flow at the same
order.

    \maketitle

    \titlecontents{section}[0pt]{\vspace{0\baselineskip}\bfseries}
    {\thecontentslabel\quad}{}%
    {\hspace{0em}\titlerule*[10pt]{$\cdot$}\contentspage}

    \titlecontents{subsection}[1em]{\vspace{0\baselineskip}}
    {\thecontentslabel\quad}{}%
    {\hspace{0em}\titlerule*[10pt]{$\cdot$}\contentspage}

    \titlecontents{subsubsection}[2em]{\vspace{0\baselineskip}}
    {\thecontentslabel\quad}{}%
    {\hspace{0em}\titlerule*[10pt]{$\cdot$}\contentspage}


    \vskip .3in


\section{Introduction}

Certain physical processes in industrial applications such as
rotating machinery cooling (see \cite{Becker1962}) and industrial
wastewater purification (see \cite{Ollis1991}) can be effectively
modeled as Taylor-Couette  flow, as illustrated in Figure
\ref{fig:fig1.jpg}. This classical flow configuration describes the
steady annular motion of viscous fluid confined between two
infinitely long coaxial rotating cylinders. The inherent complex
instability characteristics of this flow, coupled with its
significant impact on fluid mixing and transport efficiency, render
the investigation of its underlying destabilization mechanisms both
theoretically valuable and practically relevant. The instability of
rotational-axial composite flows arises from competing physical
destabilization mechanisms, a characteristic that has attracted
considerable academic attention. For typical flow configurations,
researchers have systematically investigated various systems
including: Hagen-Poiseuille flow under combined rotational and axial
pressure gradients \cite{Salwen1981}, helical Couette flow with
rotational and axial slip characteristics \cite{Meseguer2000}, and
spiral Poiseuille flow under composite pressure gradients
\cite{Meseguer2002}. A profound understanding of the asymptotic
properties of Taylor-Couette flow solutions holds significant
application value in diverse fields such as desalination,
magnetohydrodynamics, and viscosity measurement.
\begin{figure}[htbp]
  \centering
  \includegraphics[width=8.0139cm]{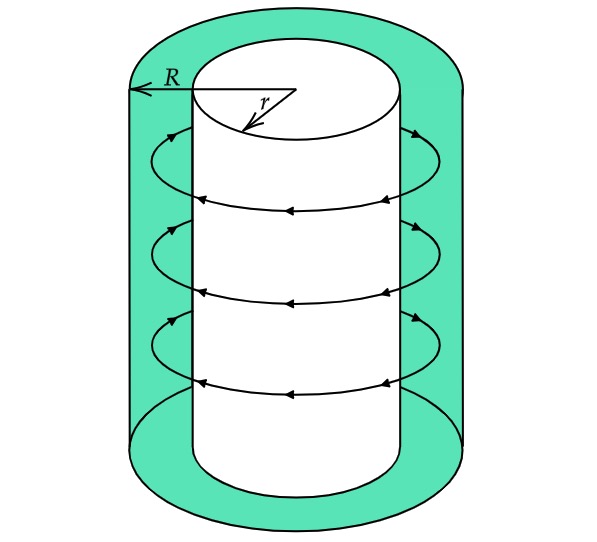} \\
  \caption{Schematic diagram of Taylor-Couette flow geometry}
  \label{fig:fig1.jpg}
\end{figure}
Despite its recognition as a fundamental rotational flow, the
perturbation analysis of Taylor-Couette flow poses significant
unresolved challenges, even despite extensive experimental,
theoretical, and numerical investigations. Many critical issues in
this field remain unresolved, sustaining its status as a focal point
in fluid mechanics research \cite{Chossat1994,Forbes2018,Pope2009}.
Notably, substantial gaps persist in mathematical proofs even for
two dimensional cases, underscoring the urgent need for fundamental
theoretical investigations. An and Li \cite{An.2024} recently
demonstrated  the stability of two dimensional Taylor-Couette flow.
Current theoretical frameworks predominantly rely on isothermal
assumptions and neglect gravitational effects, exhibiting
substantial discrepancies from realistic industrial conditions.
Particularly in rotating mechanical systems with thermal gradients,
the coupled interaction between buoyancy driven thermal convection
and centrifugal forces fundamentally alters flow stability
characteristics. This study systematically incorporates the
Boussinesq approximation to systematically investigate the critical
instability conditions and flow regime evolution under synergistic
thermal buoyancy and rotational effects, providing new theoretical
foundations for optimizing industrial equipment design.

The Boussinesq approximation captures essential thermo-hydrodynamic interactions in buoyant fluids, 
providing a fundamental framework for analyzing convection-driven phenomena. 
Its broad applicability spans atmospheric dynamics, oceanic circulation, 
and Rayleigh-Bénard convection studies \cite{Constantin.1996,Doering.1995,Majda.2003}, 
making it a cornerstone of geophysical fluid dynamics. 
The mathematical richness of the Boussinesq equations has stimulated significant research interest, 
particularly regarding two central challenges: 
(1) global well-posedness and (2) nonlinear stability thresholds. 
These interconnected problems drive ongoing theoretical advances in partial differential equations.

The dimensionless 2D Boussinesq system considered here is governed by:
\begin{equation}\label{1.1}
\begin{cases}
\partial_t v + v \cdot \nabla v = \nu \Delta v - \nabla p - \rho \mathbf{g} & \\
\partial_t \rho + v \cdot \nabla \rho = \mu \Delta \rho & \\
\nabla \cdot v = 0 & \\
\rho(0,x,y) = \rho_0(x,y), \quad v(0,x,y) = v_0(x,y)
\end{cases}
\end{equation}
where $v = (v^1, v^2)$ denotes the velocity field, 
$\rho$ represents temperature perturbation from equilibrium, 
and $\mathbf{g} = (0,g)$ is the gravitational acceleration vector. 
The parameters $\nu$ (kinematic viscosity) and $\mu$ (thermal diffusivity) 
define inverse Péclet numbers for momentum and heat transport, respectively. 
For analytical tractability, we assume $\nu = \mu$ throughout this study. 
Global well-posedness for this system is established in \cite{Abidi2007,Adhikari2016,Hou2005}, 
with solution regularity depending on initial data smoothness.
 We also introduce the vorticity $\Omega$ and the
stream function $\varphi$ are defined by
$$\Omega=\nabla\times v=\partial_y v^1-\partial_x v^2,\quad v=\nabla^{\perp}\varphi=(\partial_y\varphi,-\partial_x \varphi),\quad \Delta \varphi=\Omega,$$ which satisfy
\begin{equation}\label{1.2}
        \left\{
        \begin{aligned}
            &\partial _t\Omega +v\cdot\nabla \Omega-\nu \Delta \Omega=\partial_x\rho ,\\
            &\partial_t\rho +v\cdot\nabla \rho-\nu \Delta \rho=0,\\
            &\rho(0,x,y)=\rho_{0}(x,y),\Omega(0,x,y)=\Omega_{0}(x,y).\\
        \end{aligned}
        \right.
\end{equation}
Our attention here will be focused on the stability problem. For the
convenience of the problem, we adopt the polar coordinate
representation and impose inhomogeneous Navier-slip boundary
condition. That is,
\begin{equation}\label{1.3}
    \left\{
\begin{aligned}
    &   \partial_t \Omega-\nu \left(\partial_r^2+\frac{1}{r} \partial_r+\frac{1}{r^2} \partial_\theta^2\right)\Omega+\frac{1}{r}\left(\partial_r \varphi \partial_\theta \Omega-\partial_\theta \varphi \partial_r \Omega\right)\\
    &\quad=\cos\theta\partial_{r}\rho-\frac{\sin\theta}{r}\partial_{\theta}\rho, \\
    &\partial_t \rho-\nu\left(\partial_r^2+\frac{1}{r} \partial_r+\frac{1}{r^2} \partial_\theta^2\right) \rho+\frac{1}{r}\left(\partial_r \varphi \partial_\theta \rho-\partial_\theta \varphi \partial_r \rho\right)=0, \\
&   \left(\partial_r^2+\frac{1}{r} \partial_r+\frac{1}{r^2} \partial_\theta^2\right) \varphi=\Omega,\\
&\left (\Omega,\rho)\right|_{r=1, R}=(2A,1)
\text { with }(r, \theta) \in[1, R] \times \mathbb{S}^1 \text { and } t \geq 0.
    \end{aligned}
\right.
\end{equation}
These boundary conditions emerge from the physical scenario of a fluid obeying Navier slip conditions at solid interfaces, 
with constant tangential stress imposed on boundaries. 
Kinetic theory derivations of such boundary constraints are detailed in \cite{Masmoudi.202301}. The central objective of this work is to characterize the nonlinear stability properties 
and asymptotic dynamics of small-amplitude disturbances superposed on Taylor-Couette flow. 
This canonical flow configuration is described by the exact solution:
\begin{align*}
\mathbf{U}(r,\theta) &= \left(Ar + \frac{B}{r}\right) (-\sin\theta, \cos\theta),\\
\widetilde{\Omega}(r,\theta) &= 2A, \\
\rho_s(r,\theta) &= 1
\end{align*}
which satisfies the governing system \eqref{1.1} for arbitrary rotation parameters $A$ and $B$. 
Here $A$ governs rigid-body rotation while $B$ represents potential flow effects, 
collectively determining the vorticity distribution.
Now, we introduce  the perturbation
$$u=v-\mathbf{U}(r,\theta),\quad \omega=\Omega-\tilde{\Omega}(r,\theta),\quad \tilde{\rho}(r,\theta)=\rho(r,\theta)-\rho_s(r,\theta)$$ and  for notational convenience, we shall write $\rho$ for $\tilde{\rho}$ from now on, then
$(\varphi,\rho,\omega)$ satisfies
\begin{equation}\label{1.3}
    \left\{
\begin{aligned}
    &   \partial_t \omega-\nu \left(\partial_r^2+\frac{1}{r} \partial_r+\frac{1}{r^2} \partial_\theta^2\right)\omega+\left(A+\frac{B}{r^2}\right) \partial_\theta\omega+\frac{1}{r}\left(\partial_r \varphi \partial_\theta \omega-\partial_\theta \varphi \partial_r \omega\right)\\
    &\quad=\cos\theta\partial_{r}\rho-\frac{\sin\theta}{r}\partial_{\theta}\rho, \\
    &   \partial_t \rho-\nu\left(\partial_r^2+\frac{1}{r} \partial_r+\frac{1}{r^2} \partial_\theta^2\right) \rho+\left(A+\frac{B}{r^2}\right) \partial_\theta \rho+\frac{1}{r}\left(\partial_r \varphi \partial_\theta \rho-\partial_\theta \varphi \partial_r \rho\right)=0, \\
&   \left(\partial_r^2+\frac{1}{r} \partial_r+\frac{1}{r^2} \partial_\theta^2\right) \varphi=\omega,\\
&\left (\omega,\rho)\right|_{r=1, R}=0
\text { with }(r, \theta) \in[1, R] \times \mathbb{S}^1 \text { and } t \geq 0.
    \end{aligned}
\right.
\end{equation}

We focus primarily on the study of the transition threshold problem.
Given the norms $\|\cdot\|_{X_1}$ and $\|\cdot\|_{X_2}$, our objective is to determine $\alpha=\alpha\left(X_1, X_2\right)$ and $\beta=\beta\left(X_1, X_2\right)$ such that
$$
\begin{aligned}
& \left\|u_0\right\|_{X_1} \leq \nu^\alpha \text { and }\left\|\rho_0\right\|_{X_2} \leq \nu^\beta \Rightarrow \text { stability; } \\
& \left\|u_0\right\|_{X_1} \gg \nu^\alpha \text { or }\left\|\rho_0\right\|_{X_2} \gg \nu^\beta \Rightarrow \text { instability. }
\end{aligned}
$$
Prior to presenting our key theorem, we contextualize it within the landscape of hydrodynamic stability theory for parallel flows. System \eqref{1.1} admits the well-known stationary solution:
\begin{equation}\label{1.6}
\mathbf{u}_s = (y, 0), \quad \rho_s = 1, \quad P_s = -y + C
\end{equation}
representing unidirectional shear flow. For the infinite channel domain $\mathbb{T} \times \mathbb{R}$, optimal stability thresholds were established through quasi-linearization techniques \cite{Bian.2022,Deng.2021,Niu2024,Zillinger.2021,Zillinger.20210101,Zillinger.2023}. The sharpest existing result requires:
\begin{align}
\|\mathbf{u}_0\|_{H^s} &\leq c_0 \nu^{1/3}, \\
\|\langle x \rangle \rho_0\|_{H^s} &\leq c_0 \nu^{2/3}
\end{align}
with $s > 5$ and $\nu = \mu$ \cite{Niu2024}. This methodology decomposes the governing equations into:
\begin{itemize}
\item A \textit{transport-dominant subsystem} preserving initial regularity,
\item Coupled \textit{quasi-linear equations} initialized at zero.
\end{itemize}
For bounded channels $\mathbb{T} \times I$ with no-slip boundaries, Masmoudi et al.\cite{Masmoudi.2023} derived the threshold:
\begin{align}
\|\omega_0\|_{H^1} &\leq c_0 \min\{\nu, \mu\}^{1/2}, \\
\|\rho_0\|_{H^1} + \||D_x|^{1/6} \rho_0\|_{H^1} &\leq c_0 \min\{\nu, \mu\}^{11/12}
\end{align}
demonstrating boundary effects alter stability scalings.

The core mechanisms driving flow systems back to equilibrium in hydrodynamic stability theory are fundamentally attributed to enhanced dissipation and inviscid damping. In essence, enhanced dissipation achieves stabilization through “corrective reinforcement", while inviscid damping suppresses instability via “destabilization mitigation" of nonlinear terms. These dual phenomena have been systematically elucidated in stability studies of Couette flow within the Navier-Stokes framework. The following exposition delineates their mathematical underpinnings and physical interpretations.

Shear flow fields induce structural reorganization of perturbations
through stretching folding dynamics, progressively fragmenting
low-frequency large-scale vortices into high-frequency small-scale
disturbances. Given that the decay rate of viscous diffusivity $\nu$
scales quadratically with wavenumber (manifested as $e^{-\nu k^2
t}$), high-frequency perturbation energy exhibits exponentially
accelerated attenuation under diffusive effects. This energy cascade
from low-to-high frequency modes results in significantly enhanced
overall dissipation efficiency compared to uniform flow conditions,
hence termed enhanced dissipation. The physical community has
experimentally observed and theoretically investigated this
phenomenon (see \cite{Yaglom.2012}). Recent mathematical
breakthroughs include the pioneering work by Constantin, Kiselev,
Ryzhik, and Zlatos \cite{Constantin2008}, who rigorously established
the rapid energy dissipation induced by fluid mixing through the
concept of relaxation-enhanced flows. Subsequent landmark studies
have advanced our understanding through analyses of
convection-diffusion equations and Navier-Stokes shear flow
stability (see
\cite{Bedrossian.201704,Bedrossian.2017,Bedrossian.2019,Bedrossian.2020,Bedrossian.2016,Bedrossian.2018,Masmoudi.2020,Wei.2023,Wei.2020}).

In contrast, under the Eulerian framework with vanishing diffusivity ($\nu=0$), shear flows can still achieve spatial redistribution of kinetic energy through mixing effects: localized perturbation energy gradually disperses throughout the flow domain, manifesting as algebraic decay of velocity perturbations. This constitutes the inviscid damping phenomenon. Early recognition dates back to Orr's 1907 investigation \cite{Orr.w1907.} on Couette flow perturbations, which revealed velocity field homogenization despite $L^2$ energy conservation. Theoretically, Bedrossian and Masmoudi \cite{Bedrossian2015} established  the proof of inviscid damping and nonlinear stability for Couette flow in 2D Euler equations. State-of-the-art developments in this field are documented in \cite{Ionescu2020,Ionescu.2023,Masmoudi2024}. When $\rho \equiv 0$, system \eqref{1.2} degenerates to the classical 2D Navier-Stokes equations. Significant theoretical advances have established rigorous stability criteria for shear flows in two key configurations:
\begin{enumerate}
    \item \textbf{Passive scalar advection}: Stability thresholds under velocity perturbations,
    \item \textbf{Full NS dynamics}: Nonlinear stability of shear profiles.
\end{enumerate}
For  infinite channel case, seminal contributions include \cite{Bedrossian.2016,Bedrossian.2017,Bedrossian.2018,Bedrossian.2020,Masmoudi.2020,Masmoudi.2022,Wei.2021,Wei.2023}, while bounded channels ($\mathbb{T}\times I$) were analyzed in \cite{Chen.2020,Chen.2024} with no-slip constraints. This body of work provides fundamental methodologies for our buoyancy-modified stability analysis.

In contrast to \cite{An.2024}, this work focuses on the interplay between temperature gradients and gravitational effects in the stability of Taylor-Couette flow. Notably, the thermal buoyancy term involves derivatives with respect to the radial coordinate $r$, which introduces destabilizing mechanisms within rotating flow configurations. Motivated by this observation, we aim to systematically investigate how thermal buoyancy influences both the linear stability thresholds and long-time dynamical behavior of Taylor-Couette systems. Specifically, this section will quantify the destabilizing role of thermal stratification and characterize its impact on transient growth phenomena and asymptotic stability. To state our result, we define
$${f}_=(t;r):=\dfrac{1}{2\pi}\int_{\mathbb{T}}f(t,\theta,r)d\theta,\quad f_{\neq}=f-{f}_==\sum_{k\neq 0}\hat{f}_k(t,r)e^{ik\theta},$$
denoting
the zero mode and the non-zero mode.
The Fourier transform of $f$ in the $\theta$ direction is defined by
$$\hat{f}_k(t,r)=\dfrac{1}{2\pi}\int_{\mathbb{T}}f(t,\theta,r)e^{-ik\theta}d\theta,$$
and $k$ is the wave number.
  Our main result is stated as follows.

\begin{theorem}\label{thm1}
Consider solutions $(\omega,\rho)$ to system \eqref{1.3} initiating from $(\omega_0,\rho_0)$. Under the scaling constraint $\log R = \mathcal{O}(\nu^{-1/3}|B|^{1/3})$, there exist absolute constants $\nu_0, c', \epsilon_0, \epsilon_1, C > 0$ (independent of $\nu, B, k, R$) such that if the initial perturbations satisfy:
\begin{align}
R\|\omega_0\|_{L_\theta^2 H_r^1} + R^{-2}(\log R)^{-3/2}\|r^2\omega_0\|_{L_\theta^2 L_r^2} + R^3\left\|\frac{\omega_0}{r^3}\right\|_{L_\theta^2 L_r^2} &\leq \epsilon_0 \nu^{1/2} |B|^{1/2} R^{-2}, \label{cond1} \\
\|\rho_0\|_{H^1} &\leq \epsilon_1 \nu^{7/6} |B|^{5/6} R^{-3} \label{cond2}
\end{align}
 for some sufficiently small $\epsilon_0,\epsilon_1,0<\ \nu\leq \nu_0 ,$ then the solution $(\omega,\rho)$ is global in time and satisfies the following stability estimate:
\begin{equation}
\begin{aligned}
\sum_{k \in \mathbb{Z}} E_k &\leq C \epsilon_0 \nu^{1/2} |B|^{1/2} R^{-2}, \\
\sum_{k \in \mathbb{Z}} H_k &\leq C \epsilon_1 \nu^{7/6} |B|^{5/6} R^{-3}.
\end{aligned}
\end{equation}
The energy functionals are defined as:
\begin{equation}
\begin{aligned}
E_k &:= \begin{cases} 
\|\mathcal{E}_k(\omega_k)\|_{L^\infty L^2} + \nu^{1/6}|k|^{1/3}|B|^{1/3} R^{-1} \|\mathcal{E}_k(\omega_k)\|_{L^2 L^2} & \\
\quad + |B|^{1/2}|k|^{3/2}R^{-2} \|\mathcal{E}_k(r^{-1/2}\varphi_k)\|_{L^2 L^\infty} & \\
\quad + (\nu k^2)^{1/2} \|\mathcal{E}_k(r^{-1} w_k)\|_{L^2 L^2}, & k \neq 0 \\
\|\omega_=\|_{L^\infty L^2}, & k = 0
\end{cases} \\
H_k &:= \begin{cases} 
\|\mathcal{E}_k(\rho_k)\|_{L^\infty L^2} + \nu^{1/6}|k|^{1/3}|B|^{1/3} R^{-1} \|\mathcal{E}_k(\rho_k)\|_{L^2 L^2} & \\
\quad + (\nu k^2)^{1/2} \|\mathcal{E}_k(r^{-1} \rho_k)\|_{L^2 L^2}, & k \neq 0 \\
\|\rho_=\|_{L^\infty L^2}, & k = 0
\end{cases}
\end{aligned}
\end{equation}
where $\mathcal{E}_k(f) := e^{c'(\nu k^2)^{1/3}|B|^{2/3} R^{-2} t} f$.
\end{theorem}
\begin{remark}
Theorem \ref{thm1} establishes \textit{the first rigorous stability framework} for Taylor-Couette flow under Boussinesq dynamics. Our threshold conditions:
\begin{equation}
\begin{aligned}
I_\omega &:= R\|\omega_0\|_{L_\theta^2 H_r^1} + R^{-2}(\log R)^{-2/3}\|r^2\omega_0\|_{L_\theta^2 L_r^2} + R^3\left\|\frac{\omega_0}{r^3}\right\|_{L_\theta^2 L_r^2} \leq \epsilon_0\nu^{1/2}|B|^{1/2}R^{-2} \label{thresh1}, \\
I_\rho &:= \|\rho_0\|_{H^1} \leq \epsilon_1\nu^{7/6}|B|^{5/6}R^{-3}. 
\end{aligned}
\end{equation}
are likely suboptimal compared to:
\begin{itemize}
\item Shear flow Boussinesq systems \cite{Bian.2022,Deng.2021,Niu2024,Zillinger.2021,Zillinger.20210101,Zillinger.2023},
\item Isothermal Taylor-Couette flow \cite{An.2024}.
\end{itemize}
The degradation ($\nu^{7/6}$ vs. $\nu^{2/3}$) stems from \textit{rotational-buoyancy coupling}:
\begin{enumerate}
\item Buoyancy term $\cos\theta\partial_r\rho$ becomes destabilizing under Taylor-Couette mixing,
\item Thermal-rotational coupling introduces new instability mechanisms.
\end{enumerate}
This necessitates enhanced dissipation ($\nu^{7/6}$) to counteract buoyancy-driven instability.
\end{remark}

\begin{remark}
The functional estimates reveal distinct stabilization mechanisms:
\begin{itemize}
\item $|B|^{1/2}|k|^{3/2}R^{-2}$ terms: \textit{Inviscid damping},
\item $\nu^{1/6}|k|^{1/3}|B|^{1/3} R^{-1}$ terms: \textit{Enhanced dissipation}.
\end{itemize}
\end{remark}
We conclude with key notation used throughout:
\begin{align*}
\|f\|_{L_t^p L_r^q} &:= \left\| \|f(t,\cdot)\|_{L_r^q(1,R)} \right\|_{L_t^p(\mathbb{R}_+)}, \\
a \sim b &:\Leftrightarrow \exists C>0 \text{ s.t. } C^{-1}b \leq a \leq Cb, \\
a \lesssim b &:\Leftrightarrow \exists C>0 \text{ s.t. } a \leq Cb,
\end{align*}
where $|k| \geq 1$, and $C > 0$ denotes generic constants independent of $\nu, k$. Subsequent sections develop:
\begin{enumerate}
\item \textit{Resolvent estimates} for linearized operators (Section 4),
\item \textit{Space-time analysis} of linearized Boussinesq dynamics via resolvent methods (Section 5),
\item \textit{Nonlinear stability proof} (Section 6).
\end{enumerate}

  \section{Foundational Estimates}
This section adapts and refines essential elliptic estimates from \cite{An.2024} for our framework. The following lemmas establish crucial relationships between stream functions and vorticity in cylindrical coordinates.

\begin{lemma}[Weighted elliptic estimates for azimuthal modes] \label{lem2.1}
For $|k| \geq 1$, let $\varphi$ satisfy $\left(\partial_r^2 - \frac{k^2 - 1/4}{r^2}\right) \varphi = \omega$ with $\varphi|_{r=1,R} = 0$. The following inequalities hold:
\begin{align}
\|\partial_r \varphi\|_{L^2}^2 + |k|^2 \|r^{-1}\varphi\|_{L^2}^2 &\lesssim |\langle \omega, \varphi\rangle| \lesssim |k|^{-2} \|r \omega\|_{L^2}^2,\\
\|\partial_r \varphi\|_{L^2}^2 + |k|^2 \|r^{-1}\varphi\|_{L^2}^2 &\lesssim |k|^{-1} \|r^{1/2} \omega\|_{L^1}^2,\\
\|r^{1/2} \partial_r \varphi\|_{L^{\infty}} + |k| \|r^{-1/2} \varphi\|_{L^{\infty}} &\lesssim \left(\frac{R}{R-1}\right)^{1/2} |k|^{-1/2} \|r \omega\|_{L^2}.
\end{align}
\end{lemma}
\begin{lemma}[Axisymmetric mode estimates] \label{lem2.2}
For $k=0$, solutions to $\left(\partial_r^2 + r^{-1}\partial_r\right) \varphi = \omega$ with $\varphi|_{r=1,R} = 0$ satisfy:
\begin{equation}
\|\partial_r \varphi\|_{L^\infty} \lesssim \left(\frac{R}{R-1}\right)^{1/2} (1 + \log R) \|r^{3/2} \omega\|_{L^2}
\end{equation}
\end{lemma}
\begin{proof}
While detailed proofs appear in \cite{An.2024}, and we omit the detail proof, the core methodology involves:
\begin{enumerate}
\item Weighted energy estimates with radial Sobolev embeddings
\item Bessel function analysis for azimuthal Fourier modes
\item Boundary-aware Hardy-type inequalities exploiting annular geometry
\end{enumerate}
Key modifications include our cylindrical coordinate framework and explicit dependence on aspect ratio $(R-1)^{-1}$.
\end{proof}

 \section{Derivation of the Perturbative Equation }
The linearized 2D  Boussinesq equation around the Taylor-Couette  flow take the form
\begin{equation}\label{3.1}
    \left\{
\begin{aligned}
    &   \partial_t \omega-\nu \left(\partial_r^2+\frac{1}{r} \partial_r+\frac{1}{r^2} \partial_\theta^2\right)\omega+\left(A+\frac{B}{r^2}\right) \partial_\theta \omega+\frac{1}{r}\left(\partial_r \varphi \partial_\theta \omega-\partial_\theta \varphi \partial_r \omega\right)\\
    &\quad=\cos\theta\partial_{r}\rho-\frac{\sin\theta}{r}\partial_{\theta}\rho, \\
    &   \partial_t \rho-\nu\left(\partial_r^2+\frac{1}{r} \partial_r+\frac{1}{r^2} \partial_\theta^2\right) \rho+\left(A+\frac{B}{r^2}\right) \partial_\theta \rho+\frac{1}{r}\left(\partial_r \varphi \partial_\theta \rho-\partial_\theta \varphi \partial_r \rho\right)=0, \\
&   \left(\partial_r^2+\frac{1}{r} \partial_r+\frac{1}{r^2} \partial_\theta^2\right) \varphi=\omega,\\
&\left (\omega,\rho)\right|_{r=1, R}=0 \quad \text { with }(r, \theta) \in[1, R] \times \mathbb{S}^1 \text { and } t \geq 0 .
    \end{aligned}
\right.
\end{equation}
Applying a Fourier decomposition in the $\theta$ direction, we express the perturbation fields as:
\[
\omega(t,r,\theta) = \sum_{k\in\mathbb{Z}} \hat{\omega}_k(t,r) e^{ik\theta}, \quad \rho(t,r,\theta) = \sum_{k\in\mathbb{Z}} \hat{\rho}_k(t,r) e^{ik\theta}.
\]
Substituting these expansions into \eqref{3.1} and projecting onto each Fourier mode \(k\), we obtain the coupled system \eqref{3.2}:
\begin{equation}\label{3.2}
    \left\{
    \begin{aligned}
        &   \partial_t \hat{\omega}_k-\nu\left(\partial_r^2+\frac{1}{r} \partial_r-\frac{k^2}{r^2} \right) \hat{\omega}_k+\left(A+\frac{B}{r^2}\right) ik \hat{\omega}_k\\
  &\quad+\frac{1}{r}\left[ \sum_{l \in \mathbb{Z}} i (k-l)\partial_r\hat{\varphi}_l \hat{\omega}_{k-l}- \sum_{l \in \mathbb{Z}} i l  \hat{\varphi}_l \partial_r\hat{w}_{k-l}\right]\\
        &\quad=\frac{\partial_{r}(\hat{\rho}_{k-1}+\hat{\rho}_{k+1})}{2}+ \frac{(k+1)\hat{\rho}_{k+1}-(k-1)\hat{\rho}_{k-1}}{2r},\\
        &   \partial_t \hat{\rho}_k-\nu\left(\partial_r^2+\frac{1}{r} \partial_r-\frac{k^2}{r^2} \right) \hat{\rho}_k+\left(A+\frac{B}{r^2}\right) ik \hat{\rho}_k \\
  &\quad+\frac{1}{r}\left[ \sum_{l \in \mathbb{Z}} i (k-l)\partial_r\hat{\varphi}_l \hat{\rho}_{k-l}- \sum_{l \in \mathbb{Z}} i l  \hat{\varphi}_l \partial_r\hat{\rho}_{k-l}\right],\\
        &   \left(\partial_r^2+\frac{1}{r} \partial_r-\frac{k^2}{r^2}\right) \hat{\varphi}_k=\hat{\omega}_k,\left.\quad (\hat{\omega}_k,\hat{\rho}_k)\right|_{r=1, R}=0  .
    \end{aligned}
    \right.
\end{equation}
To eliminate the first-order derivative terms $\frac{1}{r} \partial_r$, we introduce the weight $r^{\frac{1}{2}}$ and define $\omega_k,\varphi_k$ and $\rho_k$ as follows:$$\omega_k=r^{\frac{1}{2}} e^{i k A t} \hat{\omega}_k, \quad \varphi_k=r^{\frac{1}{2}} e^{i k A t} \hat{\varphi}_k,\quad \rho_k=r^{\frac{1}{2}} e^{i k A t} \hat{\rho}_k,$$ which transforms \eqref{3.2} into the following form
\begin{equation}\label{3.3}
    \left\{
    \begin{aligned}
        &\partial_t\omega_k(t,r)+\mathcal{L}_\nu{\omega}_k(t,r)+\frac{1}{r}\left[i k \sum_{l \in \mathbb{Z}} \partial_r\left(r^{-\frac{1}{2}} \varphi_l\right) \omega_{k-l}-r^{\frac{1}{2}} \partial_r\left(\sum_{l \in \mathbb{Z}} i l r^{-1} \varphi_l \omega_{k-l}\right)\right]\\
   &=\frac{e^{iAt}}{2}\left[ r^{1/2}\partial_{r}(r^{-1/2}\rho_{k-1})\right]
        +\frac{e^{-iAt}}{2}\left[ r^{1/2}\partial_{r}(r^{-1/2}\rho_{k+1})\right]\\
  &\quad+e^{iAt}\frac{k+1}{2r}\rho_{k+1}-e^{-iAt}\frac{k-1}{2r}\rho_{k-1},\\
        &\partial_t\rho_k(t,r)+\mathcal{L}_\nu\rho_k(t,r)
  +\frac{1}{r}\left[i k \sum_{l \in \mathbb{Z}} \partial_r\left(r^{-\frac{1}{2}} \varphi_l\right) \rho_{k-l}-r^{\frac{1}{2}} \partial_r\left(\sum_{l \in \mathbb{Z}} i l r^{-1} \varphi_l \rho_{k-l}\right)\right] =0, \\
    \end{aligned}
    \right.
\end{equation}
where $$\mathcal{L}_\nu=\nu\left(\partial_r^2-\frac{k^2-\frac{1}{4}}{r^2}\right) +\frac{i k B}{r^2}.$$
For the linearized equation, we consider the Navier-slip boundary condition for the velocity and the Dirichlet  boundary condition for the temperature.
Thus and for $k\neq 0,$
$$\left.\quad (\omega_k,\varphi_k,\rho_k)\right|_{r=1, R}=0.$$

The canonical approach to linear stability assessment involves spectral analysis of the linearized dynamics. We seek modal solutions of the form:
\begin{align*}
\omega_k(t,r) &= \hat{\omega}_k(r) e^{-ikB\lambda t} \\
\varphi_k(t,r) &= \hat{\varphi}_k(r) e^{-ikB\lambda t} \\
\rho_k(t,r) &= \hat{\rho}_k(r) e^{-ikB\lambda t}
\end{align*}
where circumflexes denote complex-valued radial profiles. These ansatzes transform the system into the modified Orr-Sommerfeld equations:
\begin{equation}\label{3.4}
\begin{cases}
(\mathcal{L}_\nu - ikB\lambda)\hat{\omega}_k = \dfrac{e^{iAt}}{2} r^{1/2}\partial_r(r^{-1/2}\hat{\rho}_{k-1}) + \dfrac{e^{-iAt}}{2} r^{1/2}\partial_r(r^{-1/2}\hat{\rho}_{k+1}) \\
\quad + e^{iAt}\dfrac{k+1}{2r}\hat{\rho}_{k+1} - e^{-iAt}\dfrac{k-1}{2r}\hat{\rho}_{k-1} \\[2ex]
(\mathcal{L}_\nu - ikB\lambda)\hat{\rho}_k = 0
\end{cases}
\end{equation}
Linear instability is confirmed when nontrivial solutions exist for $\lambda \in \mathbb{C}$, $k > 0$ with $\text{Im}\lambda > 0$.

The vorticity equation in \eqref{3.4} contains a distinctive buoyancy term absent in standard Navier-Stokes linearization. Crucially, the thermal equation decouples, permitting direct application of established Navier-Stokes resolvent theory. The central mathematical challenge reduces to deriving resolvent bounds for the coupled system:
\begin{equation}\label{3.5}
\begin{cases}
(\mathcal{L}_\nu - ikB\lambda)\hat{\omega}_k = \nu\left(\partial_r^2 - \frac{k^2 - 1/4}{r^2}\right) \hat{\omega}_k + ikB\left(\frac{1}{r^2} - \lambda\right)\hat{\omega}_k = F \\[1ex]
\left(\partial_r^2 - \frac{k^2 - 1/4}{r^2}\right) \hat{\varphi}_k = \hat{\omega}_k \\[1ex]
(\hat{\omega}_k,\hat{\varphi}_k)|_{r=1,R} = 0
\end{cases}
\end{equation}
Resolvent estimates for \eqref{3.5} were rigorously established in \cite{An.2024}. Building upon this foundation, we adapt their methodology to derive compatible linear estimates for velocity and vorticity fields.

Second, we analyze the linearized temperature equation derived from:
\begin{equation}\label{3.6}
    \left\{
    \begin{aligned}
        &(\mathcal{L}_\nu-ikB\lambda)\hat{\rho}_k(t,r)=\nu\left(\partial_r^2-\frac{k^2-\frac{1}{4}}{r^2}\right) \hat{\rho}_k+i k B(\frac{1}{r^2} -\lambda)\hat{\rho}_k =F,\\
        &\hat{\rho}_k|_{r=1, R}=0,
    \end{aligned}
    \right.
\end{equation}
which shares the same linearized operator as the system in
\eqref{3.5}. To establish the required space-time estimates for this
problem, we employ fundamental resolvent estimates for the operator
$\mathcal{L}_\nu-ikB\lambda$. This approach leverages the spectral
properties of the linearized operator, and the detailed
analysis-including the treatment of the singular terms involving
$r^{-2}$ and the boundary conditions are systematically developed in
Section \ref{sub4.2}.
\section{ Resolvent estimates for Orr-Sommerfeld equation}
In this section, we analyze resolvent estimates for the linearized operator, with particular emphasis on the vorticity equation due to its shared linear structure with the temperature equation. Following a Fourier transform in the temporal variable $t$, the resolvent system reduces to the modified form:
\begin{equation}\label{4.1}
    \left\{
    \begin{aligned}
        &(\mathcal{L}_\nu-ikB\lambda)\omega=\nu\left(\partial_r^2-\frac{k^2-\frac{1}{4}}{r^2}\right) \omega+i k B(\frac{1}{r^2} -\lambda)\omega =F,\\
        &   \left(\partial_r^2-\frac{k^2-\frac{1}{4}}{r^2}\right)  \varphi=\omega,\\
        &   \left. (\omega,\varphi)\right|_{r=1, R}=0.\
    \end{aligned}
    \right.
\end{equation}
The domain of the operator is defined as
{\small$$
D_k=
    \left\{\omega\in H_{l o c}^2\left(\mathbb{R}_{+}, d r\right) \cap L^2\left(\mathbb{R}_{+}, d r\right):-\nu\left(\partial_r^2-\frac{k^2-\frac{1}{4}}{r^2}\right) \omega+i \frac{k B}{r^2} \omega\in L^2\left(\mathbb{R}_{+}, d r\right)\right\}.
$$}
We further define the following function space norms:
$$
\|f\|_{H_r^1}^2:=\left\|f^{\prime}\right\|_{L^2}^2+\left\|\frac{f}{r}\right\|_{L^2}^2 \text { and }\|f\|_{H_r^{-1}}:=\sup _{\|g\|_{H_r^1} \leq 1}|\langle f, g\rangle|,
$$
where $\langle\cdot,\cdot \rangle$ denotes the standard $L^2$ inner product over $\mathbb{R}^{+}$ with respect to the measure $dr$. Within this framework, we derive resolvent estimates for the quantities $\omega^{\prime}, \dfrac{\omega}{r}, \varphi^{\prime}$ and $\dfrac{\varphi}{r}$ in relation to the forcing term $F$. These estimates are established rigorously in both the $L^2$ norm and the negative Sobolev norm $H_r^{-1}$.

\begin{proposition}[Adapted from Proposition 3.2 in \cite{An.2024}]\label{pro4.1}
Consider the solution $\omega \in D_k$ of equation \eqref{4.1} for any $|k| \geq 1$, $\lambda \in \mathbb{R}$. There exist positive constants $C$ and $c$, independent of $\nu$, $k$, $B$, $\lambda$, and $R$, such that for any $0 \leq c^{\prime} \leq c$, the following estimates hold:

1.  Vorticity Field $\omega$ Estimates:
    $$
    \begin{aligned}
    \nu^{\frac{2}{3}} |k B|^{\frac{1}{3}} \| \omega^{\prime} \|_{L^2} + \nu^{\frac{1}{3}} |k B|^{\frac{2}{3}} \left\| \frac{\omega}{r} \right\|_{L^2} \leq C \left\| r F - c^{\prime} \nu^{\frac{1}{3}} |k B|^{\frac{2}{3}} \frac{\omega}{r} \right\|_{L^2}.
    \end{aligned}
    $$

2.  Stream Function $\varphi$ Estimates:
    $$
    \begin{aligned}
    \nu^{\frac{1}{6}} |k B|^{\frac{5}{6}} |k|^{\frac{1}{2}} \left( \| \varphi^{\prime} \|_{L^2} + |k| \left\| \frac{\varphi}{r} \right\|_{L^2} \right) \leq C R^2 \left[ \left( \frac{\nu}{|k B|} \right)^{\frac{1}{6}} (\log R)^{\frac{1}{2}} + 1 \right] \cdot \\ \left\| r F - c^{\prime} \nu^{\frac{1}{3}} |k B|^{\frac{2}{3}} \frac{\omega}{r} \right\|_{L^2}.
    \end{aligned}
    $$

Furthermore, the following additional estimates are satisfied:

3. Alternative $\omega$ Estimates:
    $$
    \nu \| \omega \|_{H_r^1} + \nu^{\frac{2}{3}} |k B|^{\frac{1}{3}} \left\| \frac{\omega}{r} \right\|_{L^2} \leq C \left\| F - c^{\prime} \nu^{\frac{1}{3}} |k B|^{\frac{2}{3}} R^{-2} \omega \right\|_{H_r^{-1}}.
    $$

4. Alternative $\varphi$ Estimates:
    $$
    \nu^{\frac{1}{2}} |k B|^{\frac{1}{2}} \| \varphi^{\prime} \|_{L^2} + \nu^{\frac{1}{2}} |k| |k B|^{\frac{1}{2}} \left\| \frac{\varphi}{r} \right\|_{L^2} \leq C R^2 \left\| F - c^{\prime} \nu^{\frac{1}{3}} |k B|^{\frac{2}{3}} R^{-2} \omega \right\|_{H_r^{-1}}.
    $$
\end{proposition}
\section{Space-time estimate of the linearized Boussinesq equation}

This section develops space-time estimates for the linearized two-dimensional Boussinesq system. 
We employ Fourier series decomposition in the angular variable $\theta \in \mathbb{T}$, expressing the fields as:
\begin{equation}
\begin{aligned}
\rho(t,\theta,r) &= \sum_{k \in \mathbb{Z}} \rho_k(t,r) e^{ik\theta}, \\
\omega(t,\theta,r) &= \sum_{k \in \mathbb{Z}} \omega_k(t,r) e^{ik\theta}, \\
u(t,\theta,r) &= \sum_{k \in \mathbb{Z}} u_k(t,r) e^{ik\theta}.
\end{aligned}
\end{equation}
Throughout our analysis, we simplify notation by omitting Fourier hats in coefficient functions.
\subsection{Space-time estimates for the vorticity}
Let us first  study the following system for $k\neq0:$
\begin{equation}\label{5.1}
    \left\{
    \begin{aligned}
    &\partial_t \omega_k+\mathcal{L}_\nu \omega_k=h_1-g\partial_{r}h_2, \\
    &\left(\partial_r^2-\frac{k^2-\frac{1}{4}}{r^2}\right)  \varphi_k=\omega_k,\\
    &\left. (\omega_k,\varphi_k)\right|_{r=1, R}=0, \quad \left. \omega_k\right|_{t=0}=\omega_k(0).\\
\end{aligned}
    \right.
\end{equation}

\begin{proposition}[Based on Proposition 5.2 in \cite{An.2024}]\label{pro5.1}
Given functions $h_1(t,r)$, $h_2(t,r)$, and $g(r)$. For any $k \in \mathbb{Z}\setminus\{0\}$ satisfying $\nu k^2 \leq |B|$ and $\log R \lesssim \nu^{-\frac{1}{3}}|B|^{\frac{1}{3}}$, let $\omega_k$ solve equation \eqref{5.1} with initial data $\omega_k(0) \in L^2$. 
Define the exponential weight:
\[
\mathcal{E}_k(f) = e^{c^{\prime}(\nu k^2)^{\frac{1}{3}}|B|^{\frac{2}{3}} R^{-2} t}f.
\]
There exists a constant $c^{\prime}>0$, independent of $\nu, B, k, R$, such that the solution satisfies:
\begin{align*}
&\mathcal{N}_1 + \mathcal{N}_2 + \mathcal{N}_3 + \mathcal{N}_4 \\
&\lesssim \mathcal{I}_1 + \mathcal{I}_2 + \mathcal{I}_3 + \mathcal{F}_1 + \mathcal{F}_2
\end{align*}
where the solution norms are:
\begin{align*}
\mathcal{N}_1 &= \|\mathcal{E}_ k(\omega_k)\|_{L^\infty_t L^2_r}, \\
\mathcal{N}_2 &= (\nu k^2)^{\frac{1}{6}}|B|^{\frac{1}{3}} R^{-1}\|\mathcal{E}_k(\omega_k)\|_{L^2_t L^2_r}, \\
\mathcal{N}_3 &= \nu^{\frac{1}{2}}\|\mathcal{E}_k(\partial_r \omega_k)\|_{L^2_t L^2_r} + (\nu k^2)^{\frac{1}{2}}\left\|\mathcal{E}_k(\frac{\omega_k}{r})\right\|_{L^2_t L^2_r}, \\
\mathcal{N}_4 &= |B|^{\frac{1}{2}} R^{-2}\left(|k|\|\mathcal{E}_k(\partial_r \varphi_k)\|_{L^2_t L^2_r} + k^2\left\|\mathcal{E}_k(\frac{\varphi_k}{r})\right\|_{L^2_t L^2_r}\right),
\end{align*}
the initial data terms are:
\begin{align*}
\mathcal{I}_1 &= \|\omega_k(0)\|_{L^2} + R^{-2}(\log R)^{-\frac{3}{2}}\|r^2 \omega_k(0)\|_{L^2} + R^3\left\|\frac{\omega_k(0)}{r^3}\right\|_{L^2}, \\
\mathcal{I}_2 &= \left(\frac{\nu}{|k B|}\right)^{\frac{1}{3}} R \|\partial_r \omega_k(0)\|_{L^2}, \\
\mathcal{I}_3 &= R\left\|\frac{\omega_k(0)}{r}\right\|_{L^2}\left(\frac{\nu k^2}{|B|}\right)^{\frac{2}{3}},
\end{align*}
and the forcing terms are:
\begin{align*}
\mathcal{F}_1 &= \nu^{-\frac{1}{6}}|k B|^{-\frac{1}{3}}\|\mathcal{E} (r h_1)\|_{L^2_t L^2_r}, \\
\mathcal{F}_2 &= \nu^{-\frac{1}{2}}\|\mathcal{E} (|g| + r|\partial_r g|) h_2 \|_{L^2_t L^2_r}.
\end{align*}
\end{proposition}

Observe that when the inequality $\nu k^2R^{-2} \geq( \nu k^2)^{1/3}|B|^{2/3}R^{-2}$ holds (equivalently, when $\nu k^2 \geq  |B|$), the heat dissipation mechanism dominates over enhanced dissipation. In this regime $\nu k^2 \geq  |B|$, the subsequent space-time estimates establish stricter controls on $\omega_k$ reflecting the heightened influence of diffusive effects.
\begin{proposition} \label{pro5.2}
    For $k \in \mathbb{Z} \backslash\{0\}$ and  $\nu k^2\geq |B|$. Let $\omega_k$ be the solution to \eqref{5.1} with $\omega_k(0) \in L^2$, then there exists a constant $c^{\prime}>0$ independent of $\nu, B, k, R$, such that it holds
    $$
    \begin{aligned}
        &   \left\|\mathcal{E}_k(\omega_k)\right\|_{L^\infty L^2}
        +\nu^{\frac{1}{2}}\left\|\mathcal{E}_k(\partial_{r}\omega_k)\right\|_{L^2L^2}\\
  &+(\nu k^2)^{\frac{1}{2}}\left\|\mathcal{E}_k(\frac{\omega_k}{r})\right\|_{L^2L^2}\\
  & \quad+(\nu k^2)^{\frac{1}{2}} R^{-2}\left(|k|\left\|\mathcal{E}_k(\partial_r \varphi_k)\right\|_{L^2 L^2}+k^2\left\|\mathcal{E}_k(\frac{\varphi_k}{r})\right\|_{L^2 L^2}\right) \\
        &\leq C\left\|\omega_k(0)\right\|_{L^2}+\nu^{-\frac{1}{2}}\left\|\mathcal{E}_k(\frac{rh_1}{k})\right\|_{L^2 L^2}\\
        &\quad+\nu^{-\frac{1}{2}}\left\|\mathcal{E}_k(|g|+r|\partial_{r}g|)h_2\right\|_{L^2L^2}.
    \end{aligned}
    $$
\end{proposition}
\begin{proof}
By performing integration by parts on the above expression and
taking its real part, one gets
$$
\begin{aligned}
& \text{Re}\left\langle\partial_t \omega_k-\nu\left(\partial_r^2-\frac{k^2-\frac{1}{4}}{r^2}\right) \omega_k+\frac{i k B}{r^2} \omega_k-h_1+g\partial_{r}h_2, w_k\right\rangle \\
&\quad =\frac{1}{2} \partial_t\left\|\omega_k\right\|_{L^2}^2+\nu\left\|\partial_r \omega_k\right\|_{L^2}^2+\nu\left(k^2-\frac{1}{4}\right)\left\|\frac{\omega_k}{r}\right\|_{L^2}^2\\
&\quad-\text{Re}\left\langle \frac{rh_1}{k}, k \frac{\omega_k}{r}\right\rangle-\text{Re}\left\langle h_2,  \partial_r\left(g \omega_k\right)\right\rangle=0 .
\end{aligned}
$$
It infers
\begin{equation}
\begin{aligned}
   & \partial_t\|\omega_k\|_{L^2}^2+\nu\left\|\partial_r \omega_k\right\|_{L^2}^2+\nu k^2\left\|\frac{\omega_k}{r}\right\|_{L^2}^2 \lesssim\left\|\frac{rh_1}{k}\right\|_{L^2}|k|\left\|\frac{\omega_k}{r}\right\|_{L^2}\\
    &+\left\|rh_2\partial_rg\right\|_{L^2}\left\|\frac{\omega_k}{r}\right\|_{L^2}+\left\|h_2g\right\|_{L^2}\left\|\partial_r\omega_k\right\|_{L^2} .
    \end{aligned}
\end{equation}
By applying Cauchy-Schwarz inequality, we then deduce
$$
\partial_t\left\|\omega_k\right\|_{L^2}^2+\nu\left\|\partial_r \omega_k\right\|_{L^2}^2+\nu k^2\left\|\frac{\omega_k}{r}\right\|_{L^2}^2 \lesssim \nu^{-1}\left(\left\|\frac{rh_1}{k}\right\|_{L^2}^2+\left\|rh_2\partial_rg\right\|_{L^2}^2+\left\|h_2g\right\|_{L^2}^2\right) .
$$
Noticing that $\left\|\dfrac{\omega_k}{r}\right\|_{L^2} \geq R^{-1}\left\|\omega_k\right\|_{L^2}$, it follows
\begin{equation}
\begin{aligned}
&\partial_t\left\|\omega_k\right\|_{L^2}^2+\nu\left\|\partial_r \omega_k\right\|_{L^2}^2+\nu k^2 R^{-2}\left\|\omega_k\right\|_{L^2}^2+\nu k^2\left\|\frac{\omega_k}{r}\right\|_{L^2}^2 \\
&\lesssim \nu^{-1}\left(\left\|\frac{rh_1}{k}\right\|_{L^2}^2+\left\|rh_2\partial_rg\right\|_{L^2}^2+\left\|h_2g\right\|_{L^2}^2\right) .
\end{aligned}
\end{equation}
Therefore, due to $ \nu k^2R^{-2} \geq( \nu
k^2)^{1/3}|B|^{2/3}R^{-2}$,  we can multiply $e^{c^{\prime}\left(\nu
k^2\right)^{\frac{1}{3}}|B|^{\frac{2}{3}} R^{-2} t}$ on both sides
of above inequality. With $c^{\prime}$ being a small constant
independent of $\nu, B, k, R$, we can obtain
$$
\begin{aligned}
& \partial_t\left\|\mathcal{E}_k (\omega_k)\right\|_{L^2}^2+\nu\left\|\mathcal{E}_k (\partial_r \omega_k)\right\|_{L^2}^2+\nu k^2\left\|\mathcal{E}_k (\frac{\omega_k}{r})\right\|_{L^2}^2 \\
&\lesssim \nu^{-1}\left\|\mathcal{E}_k (\frac{rh_1}{k})\right\|_{L^2}^2+\nu^{-1}\left\|\mathcal{E}_k(rh_2\partial_rg)\right\|_{L^2}^2\\
&\quad+\nu^{-1}\left\|\mathcal{E}_k(
h_2g)\right\|_{L^2}^2,
\end{aligned}
$$
which further implies
$$
\begin{aligned}
& \left\|\mathcal{E}_k (\omega_k)\right\|_{L^{\infty} L^2}^2
+\nu\left\|\mathcal{E}_k (\partial_r \omega_k)\right\|_{L^2 L^2}^2\\
&+\nu k^2\left\|\mathcal{E}_k (\frac{\omega_k}{r})\right\|_{L^2 L^2}^2 \\
& \lesssim\left\|\omega_k(0)\right\|_{L^2}^2
+\nu^{-1}\left\|\mathcal{E}_k(\frac{rh_1}{k})\right\|_{L^2 L^2}^2
+\nu^{-1}\left\|\mathcal{E}_k (rh_2\partial_rg)\right\|_{L^2 L^2}^2\\
&\quad+\nu^{-1}\left\|\mathcal{E}_k (h_2g)\right\|_{L^2L^2}^2 .
\end{aligned}
$$
In view of Lemma \ref{lem2.1}, it also holds that
$$
R^{-2}\left(|k|\left\|\varphi_k^{\prime}\right\|_{L^2}+k^2\left\|\frac{\varphi_k}{r}\right\|_{L^2}\right) \lesssim R^{-2}\left\|r \omega_k\right\|_{L^2} \leq\left\|\frac{\omega_k}{r}\right\|_{L^2}.
$$
Combining these two estimates above yields
$$
\begin{aligned}
& \left\|\mathcal{E}_k (\omega_k)\right\|_{L^{\infty} L^2}+\nu^{\frac{1}{2}}\left\|\mathcal{E}_k (\partial_r \omega_k)\right\|_{L^2 L^2}\\
&\quad+\left(\nu k^2\right)^{\frac{1}{2}}\left\|\mathcal{E}_k (\frac{\omega_k}{r})\right\|_{L^2 L^2} \\
&+\left(\nu k^2\right)^{\frac{1}{2}} R^{-2}\left(|k|\left\|\mathcal{E}_k (\varphi_k^{\prime})\right\|_{L^2}+k^2\left\|\mathcal{E}_k(\frac{\varphi_k}{r})\right\|_{L^2}\right) \\
& \lesssim\left\|\omega_k(0)\right\|_{L^2}
+\nu^{-\frac{1}{2}}\left\|\mathcal{E}_k (\frac{rh_1}{k})\right\|_{L^2 L^2}+\nu^{-\frac{1}{2}}\left\|\mathcal{E}_k (rh_2\partial_rg)\right\|_{L^2 L^2}\\
&\quad+\nu^{-\frac{1}{2}}\left\|\mathcal{E}_k (h_2g)\right\|_{L^2L^2},
\end{aligned}
$$
which completes the proof.
\end{proof}
\subsection{Space-time estimate for temperature}\label{sub4.2}
Next, let us consider the linearized equation
\begin{equation}\label{5.4}
    \left\{\begin{array}{l}
    \partial_t \rho_k+\mathcal{L}_\nu \rho_k=f_1-g\partial_rf_2, \\
    \left.\rho_k\right|_{t=0}=\rho_k(0)\left.\quad \rho_k\right|_{r=1, R}=0.
\end{array}\right.
\end{equation}
First, we decompose $$\rho_k=\rho_k^l+\rho_{k}^{NL},$$
where $\rho_k^{NL}$ solves
\begin{equation}\label{5.5}
    \left\{
    \begin{aligned}
        &\partial_t\rho_k^{NL}+\mathcal{L}_\nu \rho_k^{NL}=f_1-g\partial_rf_2,\\
        &\rho_k^{NL}|_{t=0}=0,\\
        &\rho_k^{NL}|_{r=1,R}=0,
    \end{aligned}
    \right.
\end{equation}
and $\rho_k^{L}$ solves
\begin{equation}\label{5.6}
    \left\{
    \begin{aligned}
        &\partial_t\rho_k^{L}+\mathcal{L}_\nu \rho_k^{L}=0,\\
        &\rho_k^{L}|_{t=0}=\rho_k(0),\\
        &\rho_k^{L}|_{r=1,R}=0.
    \end{aligned}
    \right.
\end{equation}

The homogeneous component $\rho_k^{L}$ is governed by the operator
\[
\mathcal{L}_\nu = \nu \left( \partial_r^2 - \frac{k^2 - \frac{1}{4}}{r^2} \right) + \frac{ikB}{r^2},
\]
which exhibits m-accretivity in the relevant Hilbert space framework.

An operator $H$, closed and densely defined in a Hilbert space $X$, qualifies as m-accretive when satisfying:
\begin{enumerate}
    \item $\operatorname{Re} \langle Hf, f \rangle_X \geq 0$ for all $f \in \mathcal{D}(H)$
    \item The resolvent set contains $\{\lambda \in \mathbb{C} : \operatorname{Re} \lambda < 0\}$
    \item For $\operatorname{Re} \lambda > 0$, the resolvent bound holds:
    \[
    \|(\lambda + H)^{-1}\|_{\mathcal{B}(X)} \leq (\operatorname{Re} \lambda)^{-1}
    \]
\end{enumerate}
where $\mathcal{B}(X)$ denotes the Banach algebra of bounded operators on $X$ (cf. \cite{Kato.1966,Pazy.1983}).

Additionally, we introduce the spectral gap quantity
$$\Psi(A)=\inf\left\lbrace \|(A-i\lambda)u\|:u\in D(A),\lambda \in\mathbb{R},||u||=1\right\rbrace$$
characterizing the minimal resolvent growth along the imaginary axis.

The subsequent semigroup decay estimate, adapted from \cite{Wei.202101}, applies to such operators:
\cite{Wei.202101}.
\begin{lemma}\label{lem5.3}
    Let $A$ be $m$-accretive in a Hilbert space $X$, then for any
    $t>0$, it holds that
    $$||e^{-tA}||\leq e^{-t\Psi+\frac{\pi}{2}}.$$
\end{lemma}
The Lemma \ref{lem5.3} gives the enhanced dissipation.
\begin{lemma}\label{lem5.4}
Let $\rho_k^l$ be the solution to \eqref{5.6} with $\rho_k(0) \in L^2$. Then for any $k \in \mathbb{Z}$ and $|k| \geq 1$, there exist constants $C, c>0$ being independent of $\nu, k, B, R$, such that the following inequality holds
\begin{equation}\label{5.7}
    \left\|\rho_k^l(t)\right\|_{L^2} \leq C e^{-c\left(\nu k^2\right)^{\frac{1}{3}}|B|^{\frac{2}{3}} R^{-2} t}\left\|\rho_k(0)\right\|_{L^2}, \quad \text { for any } t \geq 0 .
\end{equation}
Moreover, for any $c^{\prime} \in(0, c)$, we have
\begin{equation}\label{5.8}
\left(\nu k^2\right)^{\frac{1}{3}}|B|^{\frac{2}{3}} R^{-2}\left\|\mathcal{E}_k (\rho_k^l(t))\right\|_{L^2 L^2}^2 \leq C\left\|\rho_k(0)\right\|_{L^2}^2 .
\end{equation}
\end{lemma}
\begin{proof}
The semigroup bounds \eqref{5.7} readily follows from Lemma \ref{lem5.3} and Proposition \ref{pro4.1}. Hence for any $c^{\prime} \in(0, c)$, multiplying $\mathcal{E}_k$ on both sides of \eqref{5.7} yields
     $$
\begin{aligned}
    & 2 c^{\prime}\left(\nu k^2\right)^{\frac{1}{3}}|B|^{\frac{2}{3}} R^{-2}\left\|\mathcal{E}_k (\rho_k^l(t))\right\|_{L^2}^2 \\
    & \quad \leq 2 C c^{\prime}\left(\nu k^2\right)^{\frac{1}{3}}|B|^{\frac{2}{3}} R^{-2} e^{-2\left(c-c^{\prime}\right)\left(\nu k^2\right)^{\frac{1}{3}}|B|^{\frac{3}{3}} R^{-2} t}\left\|\rho_k(0)\right\|_{L^2}^2 .
\end{aligned}
$$
Then we integrate above inequality with respect to $t$ to get
$$
\begin{aligned}
    & 2 c^{\prime}\left(\nu k^2\right)^{\frac{1}{3}}|B|^{\frac{2}{3}} R^{-2}\left\|e^{c^{\prime}\left(\nu k^2\right)^{\frac{1}{3}}|B|^{\frac{2}{3}} R^{-2} t} \rho_k^l\right\|_{L^2 L^2}^2 \\
    & \quad \leq \int_0^{\infty} 2 C c^{\prime}\left(\nu k^2\right)^{\frac{1}{3}}|B|^{\frac{2}{3}} R^{-2} e^{-2\left(c-c^{\prime}\right)\left(v k^2\right)^{\frac{1}{3}}|B|^{\frac{2}{3}} R^{-2} t} d t\left\|\rho_k(0)\right\|_{L^2}^2 \lesssim\left\|\rho_k(0)\right\|_{L^2}^2 .
\end{aligned}
$$
This implies the estimate in \eqref{5.8}.
\end{proof}
\begin{lemma}\label{lem5.5}
Let $\rho_k^l$ be the solution to \eqref{5.6} with $\rho_k(0) \in L^2$. Then for any $k \in \mathbb{Z}$ and $|k| \geq 1$, there exist constants $C, c>0$ being independent of $\nu, k, B, R$, such that the following inequality holds
    $$
    \begin{aligned}
    &   \left\|\mathcal{E}_k(\rho_k^l)\right\|_{L^\infty L^2}^2 +\left(\nu k^2\right)^{\frac{1}{3}}|B|^{\frac{2}{3}} R^{-2}\left\|\mathcal{E}_k (\rho_k^l)\right\|_{L^2 L^2}^2\\
    &\quad+ \nu\left\|\mathcal{E}_k(\partial_{r}\rho_k^l)\right\|_{L^2L^2}^2 +\nu k^2\left\|\mathcal{E}_k(\frac{\rho_k^l}{r})\right\|_{L^2L^2}^2\\
    &\leq C\left\|\rho_k(0)\right\|_{L^2}^2.
        \end{aligned}
    $$
\end{lemma}
\begin{proof}
     We first conduct the integration by parts to get
$$
\begin{aligned}
    & \text{Re}\left\langle\partial_t \rho_k^l-\nu\left(\partial_r^2-\frac{k^2-\frac{1}{4}}{r^2}\right) \rho_k^l+\frac{i k B}{r^2} \rho_k^l, \rho_k^l\right\rangle \\
    & \quad=\frac{1}{2} \partial_t\left\|\rho_k^l\right\|_{L^2}^2+\nu\left\|\partial_r \rho_k^l\right\|_{L^2}^2+\nu\left(k^2-\frac{1}{4}\right)\left\|\frac{\rho_k^l}{r}\right\|_{L^2}^2=0 .
\end{aligned}
$$
By multiplying $e^{2 c^{\prime}\left(\nu
k^2\right)^{\frac{1}{3}}|B|^{\frac{2}{3}} R^{-2} t}$ on both sides,
one has
$$
\begin{aligned}
    & \partial_t\left\|\mathcal{E}_k (\rho_k^l)\right\|_{L^2}^2 \\
    & \quad+2 v\left(\left\|\mathcal{E}_k (\partial_r \rho_k^l)\right\|_{L^2}^2+\left(k^2-\frac{1}{4}\right)\left\|\mathcal{E}_k (\frac{\rho_k^l}{r})\right\|_{L^2}^2\right) \\
    & \leq
     \quad 2 c^{\prime}\left(\nu k^2\right)^{\frac{1}{3}}|B|^{\frac{2}{3}} R^{-2}\left\|\mathcal{E}_k (\rho_k^l)\right\|_{L^2}^2 .
\end{aligned}
$$
This implies a space-time estimate for $\rho_k^l$:
$$
\begin{aligned}
    & \left\|\mathcal{E}_k (\rho_k^l)\right\|_{L^{\infty} L^2}^2 \\
    & \quad+2 \nu\left(\left\|\mathcal{E}_k (\partial_r \rho_k^l)\right\|_{L^2 L^2}^2+\left(k^2-\frac{1}{4}\right)\left\|\mathcal{E}_k (\frac{\rho_k^l}{r})\right\|_{L^2 L^2}^2\right) \\
    & \quad \leq
    2 c^{\prime}\left(\nu k^2\right)^{\frac{1}{3}}|B|^{\frac{2}{3}} R^{-2}\left\|\mathcal{E}_k (\rho_k^l)\right\|_{L^2 L^2}^2+\left\|\rho_k(0)\right\|_{L^2}^2 .
\end{aligned}
$$
Combining with Lemma \ref{lem5.4}, we arrive at
$$
\begin{aligned}
    &   \left\|\mathcal{E}_k(\rho_k^l)\right\|_{L^\infty L^2}^2 +\left(\nu k^2\right)^{\frac{1}{3}}|B|^{\frac{2}{3}} R^{-2}\left\|\mathcal{E}_k (\rho_k^l)\right\|_{L^2 L^2}^2\\
    &\quad+ \nu\left\|\mathcal{E}_k(\partial_{r}\rho_k^l)\right\|_{L^2}^2 +\nu k^2\left\|\mathcal{E}_k(\frac{\rho_k^l}{r})\right\|_{L^2L^2}^2\\
    &\leq C\left\|\rho_k(0)\right\|_{L^2}^2.
\end{aligned}
$$
This completes the proof of Lemma \ref{lem5.5}.
\end{proof}
For the inhomogeneous part \eqref{5.5},
we have the following  lemma.
\begin{lemma}\label{lem5.6}
 Let $\rho_k^{NL}$ be the solution to \eqref{5.5}. Then for any $k \in \mathbb{Z}$ and $|k| \geq 1$, there exist constants $C, c>0$ being independent of $\nu, k, B, R$, such that the following inequality holds
    $$
 \begin{aligned}
    &   \left\|\mathcal{E}_k(\rho_k^{NL})\right\|_{L^\infty L^2}^2 +\left(\nu k^2\right)^{\frac{1}{3}}|B|^{\frac{2}{3}} R^{-2}\left\|\mathcal{E}_k (\rho_k^{NL})\right\|_{L^2 L^2}^2\\
    &\quad+ \nu\left\|\mathcal{E}_k(\partial_{r}\rho_k^{NL})\right\|_{L^2L^2}^2 +\nu k^2\left\|\mathcal{E}_k(\frac{\rho_k^{NL}}{r})\right\|_{L^2L^2}^2\\
    &\leq C\left(\nu k^2\right)^{-\frac{1}{3}}|B|^{-\frac{2}{3}}\left\|\mathcal{E}_k (rf_1)\right\|_{L^2 L^2}^2\\
  &\quad +\nu^{-1}\left\|\mathcal{E}_k((|g|+r|\partial_{r}g|f_2))\right\|_{L^2L^2}^2.
 \end{aligned}
 $$
  \end{lemma}
 \begin{proof}
    We introduce the weighted quantities
    $$\tilde{\rho}_{k}^{NL}=\mathcal{E}_k(\rho_k^{NL}), \quad\tilde{f}_j=\mathcal{E}_k(f_j),\quad j=1,2.$$
    Via a direct check, we can see that
    $$\partial_t\tilde{\rho}_k^{NL}+\left( \mathcal{L}_\nu-c^{\prime}\left(\nu k^2\right)^{\frac{1}{3}}|B|^{\frac{2}{3}} R^{-2}\right)  \tilde{\rho}_k^{NL}=\tilde{f}_1-g\partial_r\tilde{f}_2.$$
 Taking the Fourier transform in $t$
    $$\hat{\rho}(\lambda,k,r)=\int_{0}^{+\infty}\tilde{\rho}_k^{NL}e^{-it\lambda}dt,\quad F_j(\lambda,k,r)=\int_{0}^{+\infty}\tilde{f}_j(t,k,y)e^{-it\lambda}dt,\quad j=1,2,$$
    then we have
    $$\left( \mathcal{L}_\nu+i\lambda-c^{\prime}\left(\nu k^2\right)^{\frac{1}{3}}|B|^{\frac{2}{3}} R^{-2}\right)  \hat{\rho}(\lambda,k,r)=F_1-g\partial_rF_2.$$
    Using Plancherel's theorem, we know that
    $$\int_{0}^{\infty}||\tilde{\rho}^{NL}(t)||^2_{L^2}dt=\int_{\mathbb{R}}||\hat{\rho}(\lambda)||^2_{L^2}d\lambda,$$
    $$\int_{0}^{\infty}||\tilde{f}_j(t)||^2_{L^2}dt=\int_{\mathbb{R}}||F_j(\lambda)||^2_{L^2}d\lambda,\quad j=1,2.$$
    Now we use the resolvent estimates in Proposition \ref{pro4.1}  to obtain the semigroup estimates. We first decompose $\hat{\rho}(\lambda,k,y)=\hat{\rho}^{(1)}+\hat{\rho}^{(2)},$
    where $\hat{\rho}^{(1)}$ and $\hat{\rho}^{(2)}$ solve
    $$\left( \mathcal{L}_\nu+i\lambda-c^{\prime}\left(\nu k^2\right)^{\frac{1}{3}}|B|^{\frac{2}{3}} R^{-2}\right)  \hat{\rho}^{(1)}=F_1$$
    and
    $$\left( \mathcal{L}_\nu+i\lambda-c^{\prime}\left(\nu k^2\right)^{\frac{1}{3}}|B|^{\frac{2}{3}} R^{-2}\right)  \hat{\rho}^{(2)}=-g\partial_rF_2.$$
    By Proposition \ref{pro4.1},
    we have $$\nu^{\frac{2}{3}}|k B|^{\frac{1}{3}}\left\|\partial_r\hat{\rho}^1\right\|_{L^2}+\nu^{\frac{1}{3}}|k B|^{\frac{2}{3}}\left\|\frac{\hat{\rho}^1}{r}\right\|_{L^2} \leq C\left\|r F_1\right\|_{L^2},$$
    and
    $$  \nu\|\partial_r\hat{\rho}^2\|_{L^2}+\nu^{\frac{2}{3}}|k B|^{\frac{1}{3}}\left\|\frac{\hat{\rho}^2}{r}\right\|_{L^2} \leq C\left\|g\partial_rF_2\right\|_{H_r^{-1}}\leq C\left\|(|g|+r|\partial_{r}g|)F_2\right\|_{L^{2}},$$
    combining the two above inequality, then we have
    $$\nu^{\frac{2}{3}}|k B|^{\frac{1}{3}}\left\|\partial_r\hat{\rho}\right\|_{L^2}+\nu^{\frac{1}{3}}|k B|^{\frac{2}{3}}\left\|\frac{\hat{\rho}}{r}\right\|_{L^2} \leq C\left\|r F_1\right\|_{L^2}+\nu^{-\frac{1}{3}}|k B|^{-\frac{1}{3}}\left\|(|g|+r|\partial_{r}g)|F_2\right\|_{L^{2}}.$$
    According to the Plancherel's theorem, we have the following equivalence relations
    based on $L^2$ norms
  \begin{equation}\label{5.9}
      \begin{aligned}
           &\nu^{\frac{2}{3}}|k B|^{\frac{1}{3}}\left\|\partial_r\tilde{\rho}^{NL}\right\|_{L^2L^2}+\nu^{\frac{1}{3}}|k B|^{\frac{2}{3}}\left\|\frac{\tilde{\rho}^{NL}}{r}\right\|_{L^2L^2}\\
           &\leq C\left\|r \tilde{f}_1\right\|_{L^2L^2}+\nu^{-\frac{1}{3}}|k B|^{-\frac{1}{3}}\left\|(|g|+r|\partial_{r}g|)\tilde{f}_2\right\|_{L^2L^2}.
         \end{aligned}
           \end{equation}
        Applying the integration by parts, we also get
    $$
    \begin{aligned}
        & 0=\text{Re}\left\langle\partial_t \rho_{k}^{NL}-\mathcal{L}_\nu\rho_{k}^{NL}-f_1+g \partial_rf_2, e^{2 c^{\prime}\left(\nu k^2\right)^{\frac{1}{3}}|B|^{\frac{2}{3}} R^{-2} t} \rho_k^{NL}\right\rangle \\
        & =\frac{1}{2} \partial_t\left\|\mathcal{E}_k(\rho_k^{NL})\right\|_{L^2}^2+\nu\left\|\mathcal{E}_k(\partial_r \rho_k^{NL})\right\|_{L^2}^2\\
        &+\nu\left(k^2-\frac{1}{4}\right)\left\|\mathcal{E}_k (\frac{\rho_k^{NL}}{r})\right\|_{L^2}^2
        -c^{\prime}\left(\nu k^2\right)^{\frac{1}{3}}|B|^{\frac{2}{3}} R^{-2}\left\|\mathcal{E}_k (\rho_k^{NL})\right\|_{L^2}^2\\
        &-\text{Re}\left\langle f_1-g \partial_r f_2, e^{2 c^{\prime}\left(\nu k^2\right)^{\frac{1}{3}}|B|^{\frac{2}{3}} R^{-2} t} \rho_k^{NL}\right\rangle.
    \end{aligned}
    $$
    By employing the Cauchy-Schwarz inequality, we further obtain
        $$
    \begin{aligned}
        & \partial_t\left\|\mathcal{E}_k(\rho_k^{NL})\right\|_{L^2}^2\\
   &+2 \nu\left\|\mathcal{E}_k(\partial_r \rho_k^{NL})\right\|_{L^2}^2+2 \nu\left(k^2-\frac{1}{4}\right)\left\|\mathcal{E}_k (\frac{\rho_k^{NL}}{r})\right\|_{L^2}^2 \\
        & =2 c^{\prime}\left(\nu k^2\right)^{\frac{1}{3}}|B|^{\frac{2}{3}} R^{-2}\left\|\mathcal{E}_k (\rho_k^{NL})\right\|_{L^2}^2\\
   &+2 \text{Re}\left\langle f_1-g \partial_r f_2, e^{2 c^{\prime}\left(\nu k^2\right)^{\frac{1}{3}}|B|^{\frac{2}{3}} R^{-2} t} \rho_k^{NL}\right\rangle \\
        & \lesssim c^{\prime}\left(\nu k^2\right)^{\frac{1}{3}}|B|^{\frac{2}{3}} R^{-2}\left\|\mathcal{E}_k (\rho_k^{NL})\right\|_{L^2}^2\\
   &\quad+\left\|\mathcal{E}_k (r f_1)\right\|_{L^2}\left\|\mathcal{E}_k (\frac{\rho_k^{NL}}{r})\right\|_{L^2} \\
        & \quad+\| \mathcal{E}_k(\left(|g|+r\left|g^{\prime}\right| \right)f_2)\|_{L^2}\| \mathcal{E}_k (\rho_k^{NL} )\|_{H_r^1}.
    \end{aligned}
    $$
    Together with \eqref{5.9}, the above inequality yields
        $$
    \begin{aligned}
        & \left\|\mathcal{E}_k(\rho_k^{NL})\right\|_{L^{\infty} L^2}^2+\nu\left\|\mathcal{E}_k (\partial_r \rho_k^{NL})\right\|_{L^2 L^2}^2\\
   &+\nu k^2\left\|\mathcal{E}_k (\frac{\rho_k^{NL}}{r})\right\|_{L^2 L^2}^2
        +\nu^{\frac{1}{3}}|k B|^{\frac{2}{3}} \| \mathcal{E}_k(\frac{\rho_k^{NL}}{r}) \|_{L^2 L^2}^2 \\
        & \lesssim \nu^{-\frac{1}{3}}|k B|^{-\frac{2}{3}}\cdot\\
   &\left(\left\|\mathcal{E}_k (r f_1)\right\|_{L^2 L^2}+\nu^{-\frac{1}{3}}|k B|^{\frac{1}{3}} \| \mathcal{E}_k\left(|g|+r|g^{\prime}\right)  f_2 \|_{L^2 L^2}\right)^2 \\
        & \lesssim \nu^{-\frac{1}{3}}|k B|^{-\frac{2}{3}}\left\|\mathcal{E}_k (r f_1)\right\|_{L^2 L^2}^2+\nu^{-1}\left\|\mathcal{E}_k(\left(|g|+r\left|g^{\prime})\right|\right) f_2\right\|_{L^2 L^2}^2 .
    \end{aligned}
    $$
 This finishes the proof of Lemma \ref{lem5.6}.
 \end{proof}
   Combining Lemma \ref{lem5.6} and Lemma \ref{lem5.5}, we can obtain the following proposition.
 \begin{proposition}\label{pro5.7}
    For $k \in \mathbb{Z} \backslash\{0\}$ with  $\nu k^2\leq |B|$, let $\rho_k$ be the solution to \eqref{5.4} with $\rho_k(0) \in L^2$, there exists a constant $c^{\prime}>0$ independent of $\nu, B, k, R$, such that the following estimate holds:
    $$
    \begin{aligned}
        &   \left\|\mathcal{E}_k(\rho_k)\right\|_{L^\infty L^2}^2 +\left(\nu k^2\right)^{\frac{1}{3}}|B|^{\frac{2}{3}} R^{-2}\left\|\mathcal{E}_k(\rho_k)\right\|_{L^2 L^2}^2\\
        &\quad+ \nu\left\|\mathcal{E}_k(\partial_{r}\rho_k)\right\|_{L^2L^2}^2 +\nu k^2\left\|\mathcal{E}_k(\frac{\rho_k}{r})\right\|_{L^2L^2}^2\\
        &\leq C\left\|\rho_k(0)\right\|_{L^2}^2+\left(\nu k^2\right)^{-\frac{1}{3}}|B|^{-\frac{2}{3}}\left\|\mathcal{E}_k (rf_1)\right\|_{L^2 L^2}^2\\
        &\quad+\nu^{-1}\left\|\mathcal{E}_k((|g|+r|\partial_{r}g|f_2))\right\|_{L^2}^2.
    \end{aligned}
    $$
    \end{proposition}

Observe that when the inequality  $ \nu k^2R^{-2} \geq( \nu k^2)^{1/3}|B|^{2/3}R^{-2}$ holds (equivalently, when $\nu k^2 \geq  |B|$), the heat dissipation mechanism dominates over enhanced dissipation. In this regime $\nu k^2 \geq  |B|$, the subsequent space-time estimates establish stricter controls on  $\rho_k$, reflecting the heightened influence of diffusive effects.
\begin{proposition} \label{pro5.8}
    For $k \in \mathbb{Z} \backslash\{0\}$ with $\nu k^2\geq |B|$, let $\rho_k$ be the solution to \eqref{5.4} with $\rho_k(0) \in L^2$,  then there exists a constant $c^{\prime}>0$ independent of $\nu, B, k, R$, such that it holds
$$
    \begin{aligned}
        &   \left\|\mathcal{E}_k(\rho_k)\right\|_{L^\infty L^2}
        +\nu^{\frac{1}{2}}\left\|\mathcal{E}_k(\partial_{r}\rho_k)\right\|_{L^2L^2}\\
  &\quad+(\nu k^2)^{\frac{1}{2}}\left\|\mathcal{E}_k(\frac{\rho_k}{r})\right\|_{L^2L^2}\\
        &\leq C\left\|\rho_k(0)\right\|_{L^2}+\nu^{-\frac{1}{2}}\left\|\mathcal{E}_k(\frac{rf_1}{k})\right\|_{L^2 L^2}\\
        &+\nu^{-\frac{1}{2}}\left\|\mathcal{E}_k((|g|+r|\partial_{r}g|)f_2)\right\|_{L^2L^2}.
    \end{aligned}
    $$
\end{proposition}
\begin{proof}
    Similarly to Proposition \ref{pro5.2} proving, we omit the details here.
\end{proof}
\section{Nonlinear stability}
In this section, we prove Theorem \ref{thm1}. Due to the buoyancy $\partial_r \theta$ in the vorticity equation, we need to divide the frequency into $\nu k^2\geq |B|$ and $\nu k^2\leq |B|$  in order to control the buoyancy term. For the two-dimensional Boussinesq system, the global existence of smooth solution is well-known for the data $u_{0}\in H^2(\Omega),\rho_{0}\in H^1(\Omega).$ The main interest of Theorem \ref{thm1} is the stability estimate
\begin{equation}\label{6.1}
    \sum_{k\in\mathbb{Z}}E_k\leq C\epsilon_0\nu^{1/2}|B|^{1/2}R^{-2},\quad
    \sum_{k\in\mathbb{Z}}H_k\leq C\epsilon_1\nu^{7/6}|B|^{5/6}R^{-3}.
\end{equation}
where the energy functional $E_k,H_k$ are defined by
    \begin{equation}
        E_k:=   \left\{
        \begin{aligned}
  & \left\|\mathcal{E}_k (\omega_k)\right\|_{L^{\infty} L^2}+\left(\nu k^2\right)^{\frac{1}{6}}|B|^{\frac{1}{3}} R^{-1}\left\|\mathcal{E}_k (\omega_k)\right\|_{L^2 L^2}\\
      &+|B|^{1/2}|k|^{3/2}R^{-2}\left\|\mathcal{E}_k (\frac{\varphi_k}{r^{1/2}})\right\|_{L^2 L^\infty}\\
      &+\left(\nu k^2\right)^{\frac{1}{2}}\left\|\mathcal{E}_k (\frac{\omega_k}{r})\right\|_{L^2 L^2}
 ,&k\neq 0,\\
            &\|\omega_0\|_{L^\infty L^2},&k=0,\\
        \end{aligned}
        \right.
    \end{equation}
    and
    \begin{equation}
        H_k:=   \left\{
        \begin{aligned}
  &\left\|\mathcal{E}_k (\rho_k)\right\|_{L^{\infty} L^2}+\left(\nu k^2\right)^{\frac{1}{6}}|B|^{\frac{1}{3}} R^{-1}\left\|\mathcal{E}_k (\rho_k)\right\|_{L^2 L^2}\\
 &+\left(\nu k^2\right)^{\frac{1}{2}}\left\|\mathcal{E}_k (\frac{\rho_k}{r})\right\|_{L^2 L^2},
            &k\neq 0,\\
            &\|\rho_0\|_{L^\infty L^2},&k=0.\\
        \end{aligned}
        \right.
    \end{equation}
For notational simplicity, we define the initial energy
\begin{equation}
        \mathcal{M}_k(0):=  \left\{
        \begin{aligned}
 & \left\|\omega_k(0)\right\|_{L^2}+R^{-2}(\log R)^{-3/2}\left\|r^2 \omega_k(0)\right\|_{L^2}\\
 &\quad+R^3\left\|\frac{\omega_k(0)}{r^3}\right\|_{L^2}+R\|\partial_r\omega_k(0)\|_{L^2},&k\neq 0,\\
  &\|\omega_0(0)\|_{L^2},&k=0.
 \end{aligned}
        \right.
    \end{equation}
Combining these estimates with a bootstrap argument, we establish the key inequality \eqref{6.1} stated below:
\begin{proposition}\label{pro6.1}
    There holds that, for $k\neq 0,$
    \begin{equation}\label{6.5}
 \begin{aligned}
     E_k&\leq \mathcal{M}_k(0)+\nu^{-\frac{1}{2}}|B|^{-\frac{1}{2}}R^2\left(\frac{R}{R-1}\right)^{\frac{1}{2}}(1+\log R)\sum_{l \in \mathbb{Z} \backslash(0)}E_lE_{k-l}\\
  &\quad+\nu^{-\frac{2}{3}}|B|^{-\frac{1}{3}}R(H_{k+1}+H_{k-1}),
  \end{aligned}
    \end{equation}
    and
    \begin{equation}\label{6.6}
        E_0\leq ||\omega_0(0)||_{L^2}+\nu^{-\frac{1}{2}}|B|^{-\frac{1}{2}}R^2\sum_{l \in \mathbb{Z} \backslash(0)}E_lE_{-l}+\nu^{-\frac{2}{3}}|B|^{-\frac{1}{3}}R(H_{1}+H_{-1}).
    \end{equation}
    For $H_0,$ there holds that:
    \begin{equation}\label{6.7}
 H_0\leq \nu^{-\frac{1}{2}}|B|^{-\frac{1}{2}}R^2\sum_{l \in \mathbb{Z} \backslash(0)}E_lH_{-l}+\left\| \rho_0(0)\right\|_{L^2}.
    \end{equation}
    For $k\neq 0,$ there holds that:
\begin{equation}\label{6.8}
    H_k\leq \|\rho_k(0)\|_{L^2}+\nu^{-\frac{1}{2}}|B|^{-\frac{1}{2}}R^2\left(\frac{R}{R-1}\right)^{\frac{1}{2}}(1+\log R)\sum_{l \in \mathbb{Z} \backslash(0)}E_lH_{k-l}.
\end{equation}
\end{proposition}
\begin{proof}
\noindent {\bf Step 1. The estimate of \eqref{6.5}.}
Set $$E_0=\|\omega_0\|_{L^\infty L^2}=||    r^{1/2}\omega_=||_{L^\infty L^2},\quad H_0=\|\rho_0\|_{L^\infty L^2}=||r^{1/2}\rho_=||_{L^\infty L^2},$$ and for $k\neq 0,$
 \begin{equation}
 \begin{aligned}
      &E_k= \left\|\mathcal{E}_k (\omega_k)\right\|_{L^{\infty} L^2}+\left(\nu k^2\right)^{\frac{1}{6}}|B|^{\frac{1}{3}} R^{-1}\left\|\mathcal{E}_k (\omega_k)\right\|_{L^2 L^2}\\
      &\quad+|B|^{\frac{1}{2}}|k|^{\frac{3}{2}}R^{-2}\left\|\mathcal{E}_k (\frac{\varphi_k}{r^{\frac{1}{2}}})\right\|_{L^2 L^\infty}+\left(\nu k^2\right)^{\frac{1}{2}}\left\|\mathcal{E}_k (\frac{\omega_k}{r})\right\|_{L^2 L^2},
     \end{aligned}
 \end{equation}
    and
 \begin{equation}
     \begin{aligned}
 H_k= &\left\|\mathcal{E}_k (\rho_k)\right\|_{L^{\infty} L^2}+\left(\nu k^2\right)^{\frac{1}{6}}|B|^{\frac{1}{3}} R^{-1}\left\|\mathcal{E}_k (\rho_k)\right\|_{L^2 L^2}\\
 &+\left(\nu k^2\right)^{\frac{1}{2}}\left\|\mathcal{E}_k (\frac{\rho_k}{r})\right\|_{L^2 L^2}.
\end{aligned}
 \end{equation}
Denoting $$
f_1=\sum_{l \in \mathbb{Z}} \partial_r\left(r^{-\frac{1}{2}} \varphi_l\right) w_{k-l}, \quad f_2=\sum_{l \in \mathbb{Z}} i l r^{-\frac{3}{2}} \varphi_l w_{k-l},
$$ then we have
\begin{equation}
\begin{aligned}
&   \partial_t \omega_k-\nu\left(\partial_r^2-\frac{k^2-\frac{1}{4}}{r^2}\right) \omega_k+\frac{i k B}{r^2} \omega_k +\frac{1}{r}\left[i k f_1-r^{\frac{1}{2}} \partial_r\left( r^{1/2} f_2\right)\right]\\
&=\frac{e^{iAt}}{2}\left[ r^{1/2}\partial_{r}(r^{-1/2}\rho_{k-1})\right]
+\frac{e^{-iAt}}{2}\left[ r^{1/2}\partial_{r}(r^{-1/2}\rho_{k+1})\right]\\
&\quad+e^{iAt}\frac{k+1}{2r}\rho_{k+1}-e^{-iAt}\frac{k-1}{2r}\rho_{k-1}.\\
        \end{aligned}
            \end{equation}
 It follows from Proposition \ref{pro5.1} that
 \begin{equation}\label{6.12}
     \begin{aligned}
    E_k&\lesssim \mathcal{M}_k(0)+\nu^{-\frac{1}{6}}|k|^{\frac{2}{3}}|B|^{-\frac{1}{3}}\left\|\mathcal{E}_k (f_1)\right\|_{L^2 L^2}+\nu^{-\frac{1}{2}}\left\|\mathcal{E}_k (f_2)\right\|_{L^2 L^2}\\
 &+\nu^{-1/6}|kB|^{-1/3}\|\frac{k+1}{2}\mathcal{E}_k(\rho_{k+1})\|_{L^2L^2}\\
    &+\nu^{-1/6}|kB|^{-1/3}\|\frac{k-1}{2}\mathcal{E}_k(\rho_{k-1})\|_{L^2L^2}\\
 &+\nu^{-\frac{1}{2}}\left\|\mathcal{E}_k (\rho_{k+1})\right\|_{L^2 L^2}+\nu^{-\frac{1}{2}}\left\|\mathcal{E}_k (\rho_{k-1})\right\|_{L^2 L^2}.\\
    &
\end{aligned}
\end{equation}
 For the  nonlinear terms
$$
f_1=\sum_{l \in \mathbb{Z}} \partial_r\left(r^{-\frac{1}{2}} \varphi_l\right) \omega_{k-l},
$$
according to Lemma \ref{lem2.1} and Lemma \ref{lem2.2},
we get
{\small
\begin{equation}
    \begin{aligned}
     &\left\|f_1\right\|_{L^2}
 =\left\|\sum_{l \in \mathbb{Z}} \partial_r\left(r^{-\frac{1}{2}} \varphi_l\right) \omega_{k-l}\right\|_{L^2} \\
     &\leq \sum_{l \in \mathbb{Z} \setminus\left\lbrace 0, k\right\rbrace }\left\|\partial_r\left(r^{-\frac{1}{2}} \varphi_l \right)\right\|_ {L ^\infty}\left\|\omega_{k-l}\right\|_{L^2}+\left\|\partial_r\left(r^{-\frac{1}{2}} \varphi_0\right)\right\|_ {L^\infty }\left\|\omega_k\right\|_{L^2}+\left\|\partial_r\left(r^{-\frac{1}{2}} \varphi_k\right)\right\| _{L^ \infty}\left\|\omega_0\right\|_{L^2}\\
 & \lesssim \left(\frac{R}{R-1}\right)^{\frac{1}{2}} (1+\log R)\cdot\\
 &\left(\sum_{l \in \mathbb{Z} \setminus\left\lbrace 0, k\right\rbrace }|l|^{-\frac{1}{2}}\left\|r\omega_l\right\|_ {L ^2}\left\|\omega_{k-l}\right\|_{L^2}+\left\|r\omega_k\right\|_ {L^2 }\left\|\omega_0\right\|_{L^2}+\left\|r\omega_0\right\|_{L^2 }\left\|\omega_k\right\|_{L^2}\right)\\
    & \lesssim R\left(\frac{R}{R-1}\right)^{\frac{1}{2}} \cdot(1+\log R) \cdot\left(\sum_{l \in \mathbb{Z} \setminus\left\lbrace 0, k\right\rbrace }\left|l|^{-\frac{1}{2}}\|\omega_l\right\|_{L^2}\left\|\omega_{k-l}\right\|_{L^2}+\left\|\omega_0\right\|_{L^2}\left\|\omega_k\right\|_{L^2}\right).
\end{aligned}
\end{equation}}
 For the  nonlinear terms
 $$
f_2=\sum_{l \in \mathbb{Z}} i l r^{-\frac{3}{2}} \varphi_l w_{k-l},
$$
the $L^2$ norm of $f_2$ can be directly controlled by
$$
\left\|f_2\right\|_{L^2}=\left\|\sum_{l \in \mathbb{Z} \backslash(0)} \frac{l \varphi_l \omega_{k-l}}{r^{\frac{3}{2}}}\right\|_{L^2} \leq \sum_{l \in \mathbb{Z} \backslash(0)}|l|\left\|\frac{\varphi_l}{r^{\frac{1}{2}}}\right\|_{L^{\infty}}\left\|\frac{\omega_{k-l}}{r}\right\|_{L^2} .
$$
Therefore, we can obtain {\small \begin{equation}\label{6.14}
    \begin{aligned}
    &\nu^{-\frac{1}{6}}|k|^{\frac{2}{3}}|B|^{-\frac{1}{3}}\left\|\mathcal{E}_k (f_1)\right\|_{L^2 L^2}\\
     &\lesssim \nu^{-\frac{1}{6}}|k|^{\frac{2}{3}}|B|^{-\frac{1}{3}}R\left(\frac{R}{R-1}\right)^{\frac{1}{2}} \cdot(1+\log R) \cdot\\
 &\sum_{l \in \mathbb{Z} \setminus\left\lbrace 0, k\right\rbrace }|l|^{-1/2}\left\|\mathcal{E}_l(\omega_l)\right\|_{L^\infty L^2}\left\|\mathcal{E}_{k-l}(\omega_{k-l})\right\|_{L^2L^2}\\
 &+\nu^{-\frac{1}{6}}|k|^{\frac{2}{3}}|B|^{-\frac{1}{3}}R\left(\frac{R}{R-1}\right)^{\frac{1}{2}}(1+\log R)
 \left\|\omega_0\right\|_{L^\infty L^2}\left\|\mathcal{E}_k(\omega_k)\right\|_{L^2L^2} \\
&   \lesssim  \nu^{-\frac{1}{6}}|k|^{\frac{2}{3}}|B|^{-\frac{1}{3}}R \left(\frac{R}{R-1}\right)^{\frac{1}{2}}(1+\log R) \cdot \\
&    \left( \left(\nu k^2\right)^{-\frac{1}{12}}|B|^{-\frac{1}{6}} R^{1/2} (\nu k^2)^{-\frac{1}{4}}R^{1/2}E_0 E_k
 +\sum_{l \in \mathbb{Z} \backslash\{0, k\}}|l|^{-1/2}\left(\nu (k-l)^2\right)^{-\frac{1}{6}}|B|^{-\frac{1}{3}} R  E_l E_{k-l}\right) \\
&   \lesssim  R^2\left(\frac{R}{R-1}\right)^{\frac{1}{2}}(1+\log R)\left(\nu^{-\frac{1}{2}}|B|^{-\frac{1}{2}}E_0 E_k+\nu^{-\frac{1}{3}}|B|^{-\frac{2}{3}}|k|^{1/3}\sum_{l \in \mathbb{Z} \backslash 0, k\}} E_l E_{k-l}\right),
\end{aligned}
\end{equation}}
where in the last line we use the fact
$$
|k| \leq 2|l||k-l| \text { for any } k \in \mathbb{Z}, l \in \mathbb{Z} \backslash\{0, k\},
$$
and the following basic inequality $$|k|^\alpha
\leq|l|^\alpha+|k-l|^\alpha\text { for any } k, l \in \mathbb{Z}
\text { and any } \alpha \in(0,1].$$ Moreover, one gets
\begin{equation}\label{6.15}
    \begin{aligned}
        \nu^{-\frac{1}{2}}\left\|f_2\right\|_{L^2L^2}&=\nu^{-\frac{1}{2}}\left\|\sum_{l \in \mathbb{Z} \backslash(0)} \frac{l \varphi_l w_{k-l}}{r^{\frac{3}{2}}}\right\|_{L^2} \leq \nu^{-\frac{1}{2}}\sum_{l \in \mathbb{Z} \backslash(0)}|l|\left\|\frac{\varphi_l}{r^{\frac{1}{2}}}\right\|_{L^2L^{\infty}}\left\|\frac{w_{k-l}}{r}\right\|_{L^\infty L^2}\\
        &\leq \nu^{-\frac{1}{2}}|B|^{-\frac{1}{2}}R^2\sum_{l \in \mathbb{Z} \backslash(0)}E_lE_{k-l}.
\end{aligned}
\end{equation}
Thus, by \eqref{6.12}, \eqref{6.14} and \eqref{6.15}, we can obtain
that for $\nu k^2\leq |B|$
     \begin{equation}\label{6.16}
         \begin{aligned}
    E_k&\lesssim\mathcal{M}_k(0)+R^2\left(\frac{R}{R-1}\right)^{\frac{1}{2}}(1+\log R)\cdot\\
 &\left(\nu^{-\frac{1}{2}}|B|^{-\frac{1}{2}}E_0 E_k+\nu^{-\frac{1}{3}}|B|^{-\frac{2}{3}}|k|^{\frac{1}{3}}\sum_{l \in \mathbb{Z} \backslash 0, k\}} E_l E_{k-l}\right)\\
 &+\nu^{-\frac{1}{2}}|B|^{-\frac{1}{2}}R^2\sum_{l \in \mathbb{Z} \backslash(0)}E_lE_{k-l}+\nu^{-1/2}|B|^{-1/2}RH_{k+1}+\nu^{-\frac{1}{2}}|B|^{-\frac{1}{2}}RH_{k-1}\\
 &+\nu^{-\frac{2}{3}}|B|^{-\frac{1}{3}}RH_{k+1}+\nu^{-\frac{2}{3}}|B|^{-\frac{1}{3}}RH_{k-1}\\
 &\lesssim \mathcal{M}_k(0)+R^2\left(\frac{R}{R-1}\right)^{\frac{1}{2}}(1+\log R)\nu^{-\frac{1}{2}}|B|^{-\frac{1}{2}}\sum_{l \in \mathbb{Z} \backslash(0)}E_lE_{k-l}\\
 &\quad+\nu^{-\frac{2}{3}}|B|^{-\frac{1}{3}}R(H_{k+1}+H_{k-1}).
\end{aligned}
\end{equation}\label{6.17}
    For $\nu k^2\geq |B|,$ due to Proposition \ref{pro5.2}, we also have
 \begin{equation}\label{6.17}
        \begin{aligned}
            E_k&\leq C\left\|\omega_k(0)\right\|_{L^2}+\nu^{-1/2}\left\|\mathcal{E}_k(f_1)\right\|_{L^2 L^2}
        +\nu^{-1/2}\left\|\mathcal{E}_k(f_2)\right\|_{L^2}\\
   &+\nu^{-1/2}(\|\mathcal{E}_k(\rho_{k+1})\|_{L^2L^2}+\|\mathcal{E}_k(\rho_{k-1})\|_{L^2L^2})\\
  \end{aligned}
  \end{equation}
  Then, we obtain
  \begin{equation}
\begin{aligned}
 &E_k \leq C\left\|\omega_k(0)\right\|_{L^2}+C R\left(\frac{R}{R-1}\right)^{\frac{1}{2}}(1+\log R)\cdot\\
   &\quad\nu^{-\frac{1}{2}}\left( \left(\nu k^2\right)^{-\frac{1}{12}}|B|^{-\frac{1}{6}} R^{1/2} (\nu k^2)^{-\frac{1}{4}}R^{1/2}E_0 E_k\right)\\
 &\quad+C R\left(\frac{R}{R-1}\right)^{\frac{1}{2}}(1+\log R)\nu^{-\frac{1}{2}}\left(\sum_{l \in \mathbb{Z} \backslash\{0, k\}}|l|^{-1/2}\left(\nu (k-l)^2\right)^{-\frac{1}{6}}|B|^{-\frac{1}{3}} R  E_l E_{k-l}\right)\\
 &\quad+\nu^{-\frac{1}{2}}|B|^{-\frac{1}{2}}R^2\sum_{l \in \mathbb{Z} \backslash(0)}E_lE_{k-l}++\nu^{-2/3}|B|^{-1/3}R(H_{k+1}+H_{k-1})\\
            &\lesssim ||\omega_{k}(0)||_{L^2}+C R^2\left(\frac{R}{R-1}\right)^{\frac{1}{2}}(1+\log R)\nu^{-\frac{1}{2}}|B|^{-\frac{1}{2}}\sum_{l \in \mathbb{Z} \backslash(0)}E_lE_{k-l}\\
   &\quad+\nu^{-\frac{2}{3}}|B|^{-\frac{1}{3}}R(H_{k+1}+H_{k-1}).
        \end{aligned}
    \end{equation}
Combining the \eqref{6.16} and \eqref{6.17}, we establish the
\eqref{6.5}. 
\vskip .1in \noindent {\bf Step 2. The estimate of
\eqref{6.6}.}
    To derive the space-time estimates for the zero mode of solutions, we explicitly exploit the inherent heat dissipation structure of the governing equation and employ integration by parts techniques. Here we recall that the nonlinear perturbation system is formulated as follows:
\begin{equation}
    \left\{
\begin{aligned}
    &   \partial_t \omega-\nu \left(\partial_r^2+\frac{1}{r} \partial_r+\frac{1}{r^2} \partial_\theta^2\right) \omega+\left(A+\frac{B}{r^2}\right) \partial_\theta \omega+\frac{1}{r}\left(\partial_r \varphi \partial_\theta \omega-\partial_\theta \varphi \partial_r \omega\right)\\
    &\quad=\cos\theta\partial_{r}\rho-\frac{\sin\theta}{r}\partial_{\theta}\rho, \\
    &   \partial_t \rho-\nu\left(\partial_r^2+\frac{1}{r} \partial_r+\frac{1}{r^2} \partial_\theta^2\right) \rho+\left(A+\frac{B}{r^2}\right) \partial_\theta \rho+\frac{1}{r}\left(\partial_r \varphi \partial_\theta \rho-\partial_\theta \varphi \partial_r \rho\right)=0, \\
&   \left(\partial_r^2+\frac{1}{r} \partial_r+\frac{1}{r^2} \partial_\theta^2\right) \varphi=\omega,\\
&\left.(\omega,\rho)\right|_{r=1, R}=0 \quad \text { with }(r, \theta) \in[1, R] \times \mathbb{S}^1 \text { and } t \geq 0.
    \end{aligned}
\right.
\end{equation}
We see that the zero mode part about $\omega$ takes the form
\begin{equation}\label{6.19}
    \left\{
    \begin{aligned}
        &   \partial_t \omega_{=}-\nu\left(\partial_r^2+\frac{1}{r} \partial_r \right) \omega_{=}+\frac{1}{r}\left(\partial_r \varphi \partial_\theta \omega-\partial_\theta \varphi \partial_r \omega\right)_{=}\\
        &\quad=\frac{\partial_{r}(\hat{\rho}_{-1}+\hat{\rho}_{1})}{2}+ \frac{\hat{\rho}_{1}+\hat{\rho}_{-1}}{2r},\\
        &   \partial_t \rho_{=}-\nu\left(\partial_r^2+\frac{1}{r} \partial_r \right) \rho_{=}+\frac{1}{r}\left(\partial_r \varphi \partial_\theta \rho-\partial_\theta \varphi \partial_r \rho\right)_{=}=0\\
        &   \left(\partial_r^2+\frac{1}{r} \partial_r\right) \varphi_{=}=\omega_{=},\left.\quad (\omega_{=},\rho_{=})\right|_{r=1, R}=0.
    \end{aligned}
    \right.
\end{equation}
By introducing $\omega_0=r^{1/2}\omega_{=},$ and employing integration by parts, we deduce that
$$
\begin{aligned}
&\text{R}e\left\langle\frac{\partial_{r}(\hat{\rho}_{-1}+\hat{\rho}_{1})}{2}+ \frac{\hat{\rho}_{1}+\hat{\rho}_{-1}}{2r},r w_=\right\rangle\\
&=\text{Re}\left\langle\partial_t w_=-\nu\left(\partial_r^2+\frac{1}{r} \partial_r\right) w_=+\frac{1}{r}\left(\partial_r \varphi \partial_\theta w-\partial_\theta \varphi \partial_r w\right)_=, r w_=\right\rangle \\
& =\frac{1}{2} \partial_r\left\|r^{\frac{1}{2}} w_=(t)\right\|_{L^2}^2+\nu\left\|r^{\frac{1}{2}} \partial_r w_=\right\|_{L^2}^2+\text{Re}\left\langle\left[\partial_r\left(\varphi \partial_\theta w\right)-\partial_\theta\left(\varphi \partial_r w\right)\right]_=, w_=\right\rangle .
\end{aligned}
$$
Observe that
$$
\left[\partial_r\left(\varphi \partial_\theta w\right)\right]_==\partial_r\left(\varphi \partial_\theta w\right)_=\text { and }\left[\partial_\theta\left(\varphi \partial_r w\right)\right]_==\partial_\theta\left(\varphi \partial_r w\right)_==0 .
$$
Applying integration by parts once more, we subsequently get
$$
\begin{aligned}
& \frac{1}{2} \partial_t\left\|r^{\frac{1}{2}} \omega_=(t)\right\|_{L^2}^2+\nu\left\|r^{\frac{1}{2}} \partial_r \omega_=\right\|_{L^2}^2+\text{Re}\left(\partial_r\left(\varphi \partial_\theta \omega\right)_=, w_=\right) \\
& \quad=\frac{1}{2} \partial_t\left\|r^{\frac{1}{2}} w_=(t)\right\|_{L^2}^2+\nu\left\|r^{\frac{1}{2}} \partial_r \omega_=\right\|_{L^2}^2-\text{Re}\left\langle\left(\varphi \partial_\theta \omega\right)_=, \partial_r \omega_=\right\rangle\\
&\quad= \text{Re}\left\langle\frac{\hat{\rho}_{1}+\hat{\rho}_{-1}}{2},r \partial_r\omega_=\right\rangle.
\end{aligned}
$$
Consequently, we derive the subsequent inequality
$$
\begin{aligned}
    \partial_t\left\|r^{\frac{1}{2}}w_=(t)\right\|_{L^2}^2+2 \nu\left\|r^{\frac{1}{2}} \partial_r w_=\right\|_{L^2}^2
    &\leq 2\left\|r^{-\frac{1}{2}}(\varphi \partial_\theta w)_=\right\|_{L^2}\left\|r^{\frac{1}{2}} \partial_r w_=\right\|_{L^2}\\
    &+\|r^{\frac{1}{2}}(\hat{\rho}_{1}+\hat{\rho}_{-1})\|_{L^2}\left\|r^{\frac{1}{2}} \partial_r w_=\right\|_{L^2},
\end{aligned}
$$
which renders
$$
\partial_t\left\|r^{\frac{1}{2}} w_=(t)\right\|_{L^2}^2+\nu\left\|r^{\frac{1}{2}} \partial_r w_=\right\|_{L^2}^2 \leq \nu^{-1}\left\|\frac{\left(\varphi \partial_\theta w\right)_=}{r^{\frac{1}{2}}}\right\|_{L^2}^2+ \nu^{-1}\|r^{\frac{1}{2}}(\hat{\rho}_{1}+\hat{\rho}_{-1})\|_{L^2}^2.
$$
We proceed by integrating the inequality over the temporal variable
$t$ to yield
$$
\begin{aligned}
&\left\|r^{\frac{1}{2}} w_=(t)\right\|_{L^2}^2+\nu \int_0^t\left\|r^{\frac{1}{2}} \partial_r w_=(s)\right\|_{L^2}^2 d s \\
&\leq \nu^{-1} \int_0^t\left\|\frac{(\varphi \partial \theta w)_=(s)}{r^{\frac{1}{2}}}\right\|_{L^2}^2 d s+\left\|r^{\frac{1}{2}} w_=(0)\right\|_{L^2}^2\\
&\quad+\nu^{-1}\int_0^t \|r^{\frac{1}{2}}(\hat{\rho}_{1}+\hat{\rho}_{-1})\|_{L^2}^2ds.
\end{aligned}
$$
Recalling the precise definitions of  $\omega_k, \varphi_k$
$$
\begin{aligned}
w_k(t, r) & =r^{\frac{1}{2}} e^{i k A t} \hat{w}_k(t, r)=\frac{1}{2 \pi} \int_0^{2 \pi} r^{\frac{1}{2}} e^{i k A t} w(t, r, \theta) e^{-i k \theta} d \theta, \\
\varphi_k(t, r) & =r^{\frac{1}{2}} e^{i k A t} \hat{\varphi}_k(t, r)=\frac{1}{2 \pi} \int_0^{2 \pi} r^{\frac{1}{2}} e^{i k A t} \varphi(t, r, \theta) e^{-i k \theta} d \theta,
\end{aligned}
$$
we can further express
$$
\int_0^t\left\|\frac{\left(\varphi \partial_\theta \omega\right)=(s)}{r^{\frac{1}{2}}}\right\|_{L^2}^2 d s=\left\|\sum_{l \in \mathbb{Z} \backslash\{0\}} r^{-\frac{1}{2}} \hat{\varphi}_l(i l) \hat{w}_{-l}\right\|_{L^2 L^2}^2=\left\|\sum_{l \in \mathbb{Z} \backslash\{0\}} \frac{l \varphi_l \omega_{-l}}{r^{\frac{3}{2}}}\right\|_{L^2 L^2}^2 .
$$
By performing temporal integration on the inequality with respect to the $t$ variable, we rigorously establish that
$$
\begin{aligned}
\left\|r^{\frac{1}{2}} w_=(t)\right\|_{L^\infty L^2}^2+\nu \left\|r^{\frac{1}{2}} \partial_r w_=\right\|_{L^2L^2}^2  &\leq \nu^{-1}\left\|\sum_{l \in \mathbb{Z} \backslash\{0\}} \frac{l \varphi_l w_{-l}}{r^{\frac{3}{2}}}\right\|_{L^2 L^2}^2 +\left\|r^{\frac{1}{2}} w_=(0)\right\|_{L^2}^2\\
&+ \nu^{-1}\|r^{\frac{1}{2}}(\hat{\rho}_{1}+\hat{\rho}_{-1})\|_{L^2L^2}^2.
\end{aligned}
$$
Thus, by using \eqref{6.15}, we obtain
    $$E_0\lesssim ||\omega_0(0)||_{L^2}+\nu^{-\frac{1}{2}}|B|^{-\frac{1}{2}}R^2\sum_{l \in \mathbb{Z} \backslash(0)}E_lE_{-l}+\nu^{-2/3}|B|^{-1/3}R(H_{1}+H_{-1}).$$
\noindent {\bf Step 3. The estimate of \eqref{6.7}.}
Similarly, we can derive the evolution equation of $\rho_=,$
$$
\begin{aligned}
&0=\text{Re}\left\langle\partial_t \rho_=-\nu\left(\partial_r^2+\frac{1}{r} \partial_r\right) \rho_=+\frac{1}{r}\left(\partial_r \varphi \partial_\theta \rho-\partial_\theta \varphi \partial_r \rho\right)_=, r \rho_=\right\rangle \\
& =\frac{1}{2} \partial_r\left\|r^{\frac{1}{2}} \rho_=(t)\right\|_{L^2}^2+\nu\left\|r^{\frac{1}{2}} \partial_r \rho_=\right\|_{L^2}^2+\text{Re}\left\langle\left[\partial_r\left(\varphi \partial_\theta \rho\right)-\partial_\theta\left(\varphi \partial_r \rho\right)\right]_=, \rho_=\right\rangle.
\end{aligned}
$$
Observe that
$$
\left[\partial_r\left(\varphi \partial_\theta \rho\right)\right]_==\partial_r\left(\varphi \partial_\theta \rho\right)_=\text { and }\left[\partial_\theta\left(\varphi \partial_r \rho\right)\right]_==\partial_\theta\left(\varphi \partial_r \rho\right)_==0 .
$$
Employing integration by parts, we deduce that
$$
\begin{aligned}
& \frac{1}{2} \partial_t\left\|r^{\frac{1}{2}} \rho_=(t)\right\|_{L^2}^2+\nu\left\|r^{\frac{1}{2}} \partial_r \rho_=\right\|_{L^2}^2+\text{Re}\left(\partial_r\left(\varphi \partial_\theta \rho\right)_=, \rho_=\right) \\
& \quad=\frac{1}{2} \partial_t\left\|r^{\frac{1}{2}} \rho_=(t)\right\|_{L^2}^2+\nu\left\|r^{\frac{1}{2}} \partial_r \rho_=\right\|_{L^2}^2-\text{Re}\left\langle\left(\varphi \partial_\theta \rho\right)_=, \partial_r \rho_=\right\rangle=0.
\end{aligned}
$$
Consequently, we derive the subsequent inequality
$$
\partial_t\left\|r^{\frac{1}{2}}\rho_=(t)\right\|_{L^2}^2+2 \nu\left\|r^{\frac{1}{2}} \partial_r \rho_=\right\|_{L^2}^2 \leq 2\left\|r^{-\frac{1}{2}}(\varphi \partial_\theta \rho)_=\right\|_{L^2}\left\|r^{\frac{1}{2}} \partial_r \rho_=\right\|_{L^2},
$$
which renders
$$
\partial_t\left\|r^{\frac{1}{2}} \rho_=(t)\right\|_{L^2}^2+\nu\left\|r^{\frac{1}{2}} \partial_r \rho_=\right\|_{L^2}^2 \leq \nu^{-1}\left\|\frac{\left(\varphi \partial_\theta \rho\right)_=}{r^{\frac{1}{2}}}\right\|_{L^2}^2.
$$
We proceed to integrate the above inequality in $t$ variable and get
$$
\begin{aligned}
\left\|r^{\frac{1}{2}} \rho_=(t)\right\|_{L^2}^2+\nu\int_0^t\left\|r^{\frac{1}{2}} \partial_r \rho_=(s)\right\|_{L^2}^2 d s &\leq \nu^{-1} \int_0^t\left\|\frac{(\varphi \partial \theta \rho)_=(s)}{r^{\frac{1}{2}}}\right\|_{L^2}^2 d s+\left\|r^{\frac{1}{2}} \rho_=(0)\right\|_{L^2}^2.\\
\end{aligned}
$$
Recalling the precise definitions of $\rho_k, \varphi_k$:
$$
\begin{aligned}
\rho_k(t, r)& =r^{\frac{1}{2}} e^{i k A t} \hat{\rho}_k(t, r)=\frac{1}{2 \pi} \int_0^{2 \pi} r^{\frac{1}{2}} e^{i k A t} \rho(t, r, \theta) e^{-i k \theta} d \theta,\\
\varphi_k(t, r) & =r^{\frac{1}{2}} e^{i k A t} \hat{\varphi}_k(t, r)=\frac{1}{2 \pi} \int_0^{2 \pi} r^{\frac{1}{2}} e^{i k A t} \varphi(t, r, \theta) e^{-i k \theta} d \theta,
\end{aligned}
$$
we can further express
$$
\int_0^t\left\|\frac{\left(\varphi \partial_\theta \rho\right)_=(s)}{r^{\frac{1}{2}}}\right\|_{L^2}^2 d s=\left\|\sum_{l \in \mathbb{Z} \backslash|\theta|} r^{-\frac{1}{2}} \hat{\varphi}_l(i l) \hat{\rho}_{-l}\right\|_{L^2 L^2}^2=\left\|\sum_{l \in \mathbb{Z} \backslash|0|} \frac{l \varphi_l \rho_{-l}}{r^{\frac{3}{2}}}\right\|_{L^2 L^2}^2.
$$
By performing temporal integration on the inequality with respect to the $t$ variable, we rigorously establish that
$$
\begin{aligned}
\left\|r^{\frac{1}{2}} \rho_=(t)\right\|_{L^\infty L^2}^2+\nu \left\|r^{\frac{1}{2}} \partial_r \rho_=\right\|_{L^2L^2}^2  &\leq \nu^{-1}\left\|\sum_{l \in \mathbb{Z} \backslash|0|} \frac{l \varphi_l \rho_{-l}}{r^{\frac{3}{2}}}\right\|_{L^2 L^2}^2 +\left\|r^{\frac{1}{2}} \rho_=(0)\right\|_{L^2}^2.
\end{aligned}
$$
Thus, we obtain
    $$H_0\lesssim \nu^{-\frac{1}{2}}|B|^{-\frac{1}{2}}R^2\sum_{l \in \mathbb{Z} \backslash(0)}E_lH_{-l}+\left\|r^{\frac{1}{2}} \rho_=(0)\right\|_{L^2}.$$

 \noindent {\bf Step 4. The estimate of \eqref{6.8}.}
 Denoting
  $$
g_1=\sum_{l \in \mathbb{Z}} \partial_r\left(r^{-\frac{1}{2}} \varphi_l\right) \rho_{k-l}, \quad g_2=\sum_{l \in \mathbb{Z}} i l r^{-\frac{3}{2}} \varphi_l \rho_{k-l},
$$
we have
        $$\partial_t \rho_k-\nu\left(\partial_r^2-\frac{k^2-\frac{1}{4}}{r^2}\right) \rho_k+\frac{i k B}{r^2} \rho_k+\frac{1}{r}\left[i k g_1-r^{\frac{1}{2}} \partial_r\left(r^{1/2}g_2\right)\right] =0.$$
     It follows from Proposition \ref{pro5.7} that
  \begin{equation}\label{6.20}
  \begin{aligned}
      &H_k\leq C\left\|\rho_k(0)\right\|_{L^2}+\nu^{-\frac{1}{6}}|k|^{\frac{2}{3}}|B|^{-\frac{1}{3}}\left\|\mathcal{E}_k (g_1)\right\|_{L^2 L^2}\\
        &\quad+\nu^{-\frac{1}{2}}\left\|\mathcal{E}_k(g_2)\right\|_{L^2},
        \end{aligned}
   \end{equation}
   and due to Lemma \ref{lem2.1} and Lemma \ref{lem2.2}, then we obtain
 \begin{equation}\label{6.21}
\begin{aligned}
    &\nu^{-\frac{1}{6}}|k|^{\frac{2}{3}}|B|^{-\frac{1}{3}}\left\|\mathcal{E}_k  (g_1)\right\|_{L^2 L^2}
     \lesssim \nu^{-\frac{1}{6}}|k|^{\frac{2}{3}}|B|^{-\frac{1}{3}}\left(\frac{R}{R-1}\right)^{\frac{1}{2}}(1+\log R) \cdot\\
 &\sum_{ l\in \mathbb{Z} \setminus\left\lbrace 0, k\right\rbrace }|l|^{-\frac{1}{2}}\left\|\mathcal{E}_l(\omega_l)\right\|_{L^\infty L^2}\left\|\mathcal{E}_{k-l}(\rho_{k-l})\right\|_{L^2L^2}\\
 &+\nu^{-\frac{1}{6}}|k|^{\frac{2}{3}}|B|^{-\frac{1}{3}}\left(\frac{R}{R-1}\right)^{\frac{1}{2}}(1+\log R) \cdot\\
 &\left(\left\|r\omega_0\right\|_{L^\infty L^2}\left\|\mathcal{E}_k(\rho_k)\right\|_{L^2L^2}+\left\|r\mathcal{E}_k(\omega_k)\right\|_{L^2 L^2}\left\|\rho_0\right\|_{L^\infty L^2} \right)\\
&   \lesssim  \nu^{-\frac{1}{6}}|k|^{\frac{2}{3}}|B|^{-\frac{1}{3}}R \left(\frac{R}{R-1}\right)^{\frac{1}{2}}(1+\log R) \cdot \\
    & \left( \left(\nu k^2\right)^{-\frac{1}{12}}|B|^{-\frac{1}{6}} R^{\frac{1}{2}} (\nu k^2)^{-\frac{1}{4}}R^{\frac{1}{2}}(E_0 H_k+E_kH_0)\right)\\
 &+\nu^{-\frac{1}{6}}|k|^{\frac{2}{3}}|B|^{-\frac{1}{3}}R \left(\frac{R}{R-1}\right)^{\frac{1}{2}}(1+\log R) \cdot \\
 &\left(\sum_{l \in \mathbb{Z} \backslash\{0, k\}}|l|^{-\frac{1}{2}}\left(\nu (k-l)^2\right)^{-\frac{1}{6}}|B|^{-\frac{1}{3}} R  E_l H_{k-l}\right) . \\
&   \lesssim   R^2\left(\frac{R}{R-1}\right)^{\frac{1}{2}}(1+\log R)\cdot\\
&\left(\nu^{-\frac{1}{2}}|B|^{-\frac{1}{2}}(E_0 H_k+E_kH_0)+\nu^{-\frac{1}{3}}|B|^{-\frac{2}{3}}|k|^{\frac{1}{3}}\sum_{l \in \mathbb{Z} \backslash 0, k\}} E_l H_{k-l}\right),
\end{aligned}
\end{equation}
where in the last line we use the fact
$$
|k| \leq 2|l||k-l| \text { for any } k \in \mathbb{Z}, l \in \mathbb{Z} \backslash\{0, k\},
$$
and the following basic inequality
 $$|k|^\alpha \leq|l|^\alpha+|k-l|^\alpha\text { for any } k, l \in \mathbb{Z} \text { and any } \alpha \in(0,1].$$
Moreover, one gets
\begin{equation}\label{6.22}
    \begin{aligned}
        \nu^{-\frac{1}{2}}\left\|g_2\right\|_{L^2L^2}&=\nu^{-\frac{1}{2}}\left\|\sum_{l \in \mathbb{Z} \backslash(0)} \frac{l \varphi_l \rho_{k-l}}{r^{\frac{3}{2}}}\right\|_{L^2} \leq \nu^{-\frac{1}{2}}\sum_{l \in \mathbb{Z} \backslash(0)}|l|\left\|\frac{\varphi_l}{r^{\frac{1}{2}}}\right\|_{L^2L^{\infty}}\left\|\frac{\rho_{k-l}}{r}\right\|_{L^\infty L^2}\\
        &\lesssim \nu^{-\frac{1}{2}}|B|^{-\frac{1}{2}}R^2\sum_{l \in \mathbb{Z} \backslash(0)}E_lH_{k-l}.
\end{aligned}
\end{equation}
        Thus, by \eqref{6.20}, \eqref{6.21} and \eqref{6.22}, we can obtain that for $\nu k^2\leq |B|$
     \begin{equation}
         \begin{aligned}
    H_k&\lesssim \|\rho_k(0)\|_{L^2}+\nu^{-\frac{1}{6}}|k|^{\frac{2}{3}}|B|^{-\frac{1}{3}}\left\|\mathcal{E}_k (g_1)\right\|_{L^2 L^2}\\
 &\quad+\nu^{-\frac{1}{2}}\left\|\mathcal{E}_k(g_2)\right\|_{L^2 L^2}\\
    &\lesssim\|\rho_k(0)\|_{L^2}
 +R^2\left(\frac{R}{R-1}\right)^{\frac{1}{2}}(1+\log R)\cdot\\&\left(\nu^{-\frac{1}{2}}|B|^{-\frac{1}{2}}(E_0 H_k+E_kH_0)+\nu^{-\frac{1}{3}}|B|^{-\frac{2}{3}}|k|^{1/3}\sum_{l \in \mathbb{Z} \backslash 0, k\}} E_l H_{k-l}\right)\\
 &+\nu^{-\frac{1}{2}}|B|^{-\frac{1}{2}}R^2\sum_{l \in \mathbb{Z} \backslash(0)}E_lH_{k-l},\\
\end{aligned}
\end{equation}
then we have
$$
H_k\lesssim \|\rho_k(0)\|_{L^2}+\nu^{-\frac{1}{2}}|B|^{-\frac{1}{2}}R^2\left(\frac{R}{R-1}\right)^{\frac{1}{2}}(1+\log R)\sum_{l \in \mathbb{Z} \backslash(0)}E_lH_{k-l}.
$$
For $\nu k^2\geq |B|,$ due to Proposition \ref{pro5.8}, we also get
    \begin{equation}\label{6.24}
        \begin{aligned}
            H_k&\leq C\left\|\rho_k(0)\right\|_{L^2}+\nu^{-1/2}\left\|\mathcal{E}_k(g_1)\right\|_{L^2 L^2}
        +\nu^{-1/2}\left\|\mathcal{E}_k(g_2)\right\|_{L^2}\\
            & \lesssim||\rho_{k}(0)||_{L^2}+C R\left(\frac{R}{R-1}\right)^{\frac{1}{2}}(1+\log R)\nu^{-\frac{1}{2}}\cdot\\
   &\left( \left(\nu k^2\right)^{-\frac{1}{12}}|B|^{-\frac{1}{6}} R^{1/2} (\nu k^2)^{-\frac{1}{4}}R^{1/2}(E_0 H_k+E_kH_0)\right)\\
 &+R\left(\frac{R}{R-1}\right)^{\frac{1}{2}}(1+\log R)\left(\sum_{l \in \mathbb{Z} \backslash\{0, k\}}|l|^{-1/2}\left(\nu (k-l)^2\right)^{-\frac{1}{6}}|B|^{-\frac{1}{3}} R  E_l H_{k-l}\right)\\
 &+\nu^{-\frac{1}{2}}|B|^{-\frac{1}{2}}R^2\sum_{l \in \mathbb{Z} \backslash(0)}E_lH_{k-l}\\
            &\lesssim ||\rho_{k}(0)||_{L^2}+\nu^{-\frac{1}{2}}|B|^{-\frac{1}{2}}R^2\left(\frac{R}{R-1}\right)^{\frac{1}{2}}(1+\log R)\sum_{l \in \mathbb{Z} \backslash(0)}E_lH_{k-l}.
        \end{aligned}
    \end{equation}
\end{proof}
    Now we prove Theorem \ref{thm1}. From \eqref{6.5} and \eqref{6.6}, we deduce
    \begin{equation}\label{6.25}
        \begin{aligned}
     &\sum_{k\in Z}E_k\lesssim \sum_{k\in Z}\mathcal{M}_k(0)+\nu^{-\frac{1}{2}}|B|^{-\frac{1}{2}}R^2\left(\frac{R}{R-1}\right)^{\frac{1}{2}}(1+\log R)\sum_{k\in Z}\sum_{l \in \mathbb{Z} \backslash(0)}E_lE_{k-l}\\
     &+\sum_{k\in Z}\nu^{-2/3}|B|^{-1/3}RH_{k}.
     \end{aligned}
\end{equation}
    And by the fact
    $$\sum_{k\in\mathbb{Z}}H_k=H_0+\sum_{k\in\mathbb{Z}\setminus\left\lbrace  0\right\rbrace, \nu k^2\leq |B|}H_k+\sum_{k\in\mathbb{Z}\setminus\left\lbrace 0\right\rbrace,\nu k^2>|B|}H_k,$$
    combining \eqref{6.7} and \eqref{6.8}, we can deduce
    \begin{equation}\label{6.26}
    \sum_{k\in Z}H_k\lesssim \sum_{k\in Z}\|\rho_k(0)\|_{L^2}+\nu^{-\frac{1}{2}}|B|^{-\frac{1}{2}}R^2\left(\frac{R}{R-1}\right)^{\frac{1}{2}}(1+\log R)\sum_{k\in Z}\sum_{l \in \mathbb{Z} \backslash(0)}E_lH_{k-l}.
\end{equation}
    On the other hand, if   $$\sum_{k\in Z}||\rho_k(0)||_{L^2}\leq \epsilon_0\nu^{7/6}|B|^{5/6}R^{-3}\left(\frac{R}{R-1}\right)^{-\frac{1}{2}}(1+\log R)^{-1},$$ and $$\sum_{k\in Z}\mathcal{M}_k(0)\leq \epsilon_1\nu^{1/2}|B|^{1/2}R^{-2}\left(\frac{R}{R-1}\right)^{-\frac{1}{2}}(1+\log R)^{-1}.$$
    Thus, for $\epsilon_0,\epsilon_1$ suitably small, by bootstrap arguments, we can deduce  from \eqref{6.25} and \eqref{6.26} that
    $$\sum_{k\in\mathbb{Z}}H_k\leq C\epsilon_0\nu^{7/6}|B|^{5/6}R^{-3}\left(\frac{R}{R-1}\right)^{-\frac{1}{2}}(1+\log R)^{-1},$$ and
 $$\sum_{k\in\mathbb{Z}}E_k\leq C\epsilon_1\nu^{1/2}|B|^{1/2}R^{-2}\left(\frac{R}{R-1}\right)^{-\frac{1}{2}}(1+\log R)^{-1}.$$
    This completes the proof of Theorem \ref{thm1}.

    \textbf{Conflict of Interest}\quad The authors declared that they have no conflict of interest.
    \section{Acknowledgment}
The author thank Prof. Weike Wang and Binbin Shi for suggesting this
problem and many valuable discussions. The author is supported by
National Nature Science Foundation of China 12271357, 12331007 and
12161141004.
    \bibliography{bib}

\end{document}